\documentclass{amsart}

\usepackage{graphicx}
\usepackage{amsfonts}
\usepackage{amsmath}
\usepackage{amsthm}

\newtheorem{Thm}{Theorem}

\newtheorem{Cor}{Corollary}
\newtheorem{Lem}{Lemma}
\newtheorem{Prop}{Proposition}

\newtheorem{Def}{Definition}
\newtheorem{Rem}{Remark}

\newcommand{\q}{\mathbb{Q}}

\newcommand{\n}{\mathbb{N}}
\newcommand{\z}{\mathbb{Z}}
\newcommand{\f}{\mathbb{F}}

\begin{document}
\title{Combinatorial Aspects of Elliptic Curves}
\author{Gregg Musiker}
\thanks{This work was supported by the NSF, grant DMS-0500557}
\address{Mathematics Department, University of California, San
Diego} \email{gmusiker@math.ucsd.edu}

\date{July 20, 2007} \maketitle

\begin{abstract}
Given an elliptic curve $C$, we study here $N_k = \# C(\f_{q^k})$,
the number of points of $C$ over the finite field $\f_{q^k}$. This
sequence of numbers, as $k$ runs over positive integers, has
numerous remarkable properties of a combinatorial flavor in addition
to the usual number theoretical interpretations.  In particular we
prove that $N_k = -\mathcal{W}_k(q,-N_1)$ where $\mathcal{W}_k(q,t)$
is a $(q,t)$-analogue of the number of spanning trees of the wheel
graph. Additionally we develop a determinantal formula for $N_k$
where the eigenvalues can be explicitly written in terms of $q$,
$N_1$, and roots of unity. We also discuss here a new sequence of
bivariate polynomials related to the factorization of $N_k$, which
we refer to as elliptic cyclotomic polynomials because of their
various properties.
\end{abstract}

\tableofcontents

\maketitle  

\section{Introduction}

An interesting problem at the cross-roads between combinatorics,
number theory, and algebraic geometry, is that of counting the
number of points on an algebraic curve over a finite field. Over a
finite field, the locus of solutions of an algebraic equation is a
discrete subset, but since they satisfy a certain type of algebraic
equation this imposes a lot of extra structure beneath the surface.
One of the ways to detect this additional structure is by looking at
field extensions: the infinite sequence of cardinalities is only
dependent on a finite set of data. Specifically the number of points
over $\f_q$, $\f_{q^2}$, \dots, and $\f_{q^g}$ will be sufficient
data to determine the number of points on a genus $g$ algebraic
curve over any other algebraic field extension.  This observation
motivates the question of how the points over higher field
extensions correspond to points over the first $g$ extensions.

To see this more clearly, we specialize to the case of elliptic
curves, where $g=1$, and examine the expressions for $N_k$, the
number of points on $C$ over $\f_{q^k}$, as functions of $q$ and
$N_1$.  It follows from the well-known rationality of the zeta
function that \begin{eqnarray} \label{preCheb} N_k(q,N_1) = 1+q^k
-\alpha_1^k-\alpha_2^k,\end{eqnarray} where $\alpha_1$ and
$\alpha_2$ are the two roots of the quadratic $1-(1+q-N_1)T+qT^2.$
Additionally, we observe, see Theorem \ref{Gar}, that
\begin{eqnarray} \label{preBiv}
N_k(q,N_1)\mathrm{~are~integral~polynomials~whose~coefficients~alternate~in~sign}.\end{eqnarray}

In this paper, we use formulas arising from (\ref{preCheb}) and
(\ref{preBiv}) to connect elliptic curves to several different areas
of combinatorics. Specifically, (\ref{preCheb}) implies that the
family of polynomials $1+q^k - N_k$ are Chebyshev polynomials of the
first kind, a well-studied example of orthogonal polynomials.  In Section $4$, we describe this perspective in further detail.
Alternatively, we can interpret statement (\ref{preCheb}) as the
plethystic expression $N_k = p_k[1+q-\alpha_1-\alpha_2]$ where the
$p_k$'s are the power symmetric functions.  In summary, we exploit
both the fields of orthogonal polynomials and symmetric functions to
illustrate numerous identities involving the $N_k$'s.

Moreover, we find that the polynomial expressions for $N_k$ due to
(\ref{preBiv}) are related to a $(q,t)$-deformation of the Lucas
numbers (Theorem \ref{qtLucEll}), and also lead to a combinatorial
interpretation involving spanning trees of the wheel graph (Theorem
\ref{whser}).  Thus the aforementioned identities also indicate
properties of the Lucas numbers and spanning trees as well.

Using these new combinatorial interpretations for $N_k$, we develop further properties of this sequence,
obtaining determinantal formulas (Theorem \ref{detformu}), as well
as formulas involving a certain bivariate version of the Fibonacci
polynomials (Theorem \ref{EkasFib}).  Another surprising by-product
of our analysis is a factorization of $N_k$ into a new sequence of
polynomials, which we refer to as elliptic cyclotomic polynomials.
Both of these families of polynomials are interesting in their own
right and have numerous properties which justify their names.
We give a geometric interpretation of the elliptic cyclotomic
polynomials as Theorem \ref{geomEcyc} and close with some
combinatorial identities involving this new family of expressions.

\section{$N_k$ as an alternating sum}

\vspace{1em}The zeta function of a curve $C$ is defined to be the
exponential generating function \begin{eqnarray} \label{zetaF}
Z(C,T) = \exp\bigg(\sum_{k\geq 1} N_k {T^k \over
k}\bigg).\end{eqnarray} A result due to
Weil \cite{Weil} is that the zeta function of a curve is rational
with specific formula given as
\begin{eqnarray} \label{rationalityC} Z(C,T) =
{(1-\alpha_1T)(1-\alpha_2T)\cdots (1-\alpha_{2g}T) \over
(1-T)(1-qT)}.\end{eqnarray} Here $g$ is the genus of curve $C$, and
the numerator is sometimes written as $L(C,T)$, a degree $2g$
polynomial with integer coefficients.
Moreover when $E$ is an elliptic curve, $Z(E,T)$ can be expressed as
\begin{eqnarray*}
{1-(\alpha_1+\alpha_2)T + \alpha_1\alpha_2T^2 \over (1-T)(1-qT)}
.\end{eqnarray*} The zeta function of a curve also satisfies a
functional equation which in the elliptic case is simply equivalent
to
$$\alpha_1\alpha_2 = q.$$

Among other things, (\ref{zetaF}) and (\ref{rationalityC}) imply
that $N_k = 1 + q^k - \alpha_1^k - \alpha_2^k-\dots -
\alpha_{2g}^k$, which can be written in plethystic notation as
$p_k[1+q-\alpha_1-\alpha_2]$.  We describe symmetric functions and
plethystic notation in more depth in Section \ref{bifibsec}. In the
case that $E$ is a curve of genus one and $k=1$ we get
$$\alpha_1 + \alpha_2 = 1 + q - N_1.$$  Hence we can rewrite the zeta function
$Z(E,T)$ totally in terms of $q$ and $N_1$ and as a consequence, all
the $N_k$'s are actually dependent on these two quantities.  The
first few formulas are given below.
\begin{eqnarray*}
N_2 &=& (2+2q)N_1 - N_1^2 \\
N_3 &=& (3+3q+3q^2)N_1 - (3+3q)N_1^2 + N_1^3 \\
N_4 &=& (4+4q+4q^2+4q^3)N_1 - (6+8q+6q^2)N_1^2 + (4+4q)N_1^3 - N_1^4 \\
N_5 &=& (5+5q+5q^2+5q^3+5q^4)N_1 - (10+15q+15q^2+10q^3)N_1^2 \\
&+& (10+15q+10q^2)N_1^3 - (5+5q)N_1^4+N_1^5
\end{eqnarray*}
This data gives rise to the following observation of Adriano Garsia.

\begin{Thm} \label{Gar}
$$N_k = \sum_{i=1}^k (-1)^{i-1}P_{i,k}(q)N_1^i$$
where the $P_{i,k}$'s are polynomials with \emph{positive integer}
coefficients.
\end{Thm}

This theorem is proved by Garsia using induction and the fact that
the sequence of $N_k$'s satisfy a simple recurrence.  For the
details, see \cite[Chap. 7]{Gars}.  This result motivates the
combinatorial question: what are the objects that the family of
polynomials, $\{P_{i,k}\}$, enumerate?  We answer this question in
due course in multiple ways, thus providing an alternate,
combinatorial, proof of Theorem \ref{Gar}.

\label{Nkexpanse}

\subsection{The Lucas numbers and a $(q,t)$-analogue} \label{LucSec}
\begin{Def}
Let $S_1^{(n)}$ be the circular shift of set $S \subseteq
\{1,2,\dots, n\}$ modulo $n$, i.e. element $x \in S_1^{(n)}$ if and
only if $x-1 ~(\mod~ n~)\in S$.  We define the {\bf $(q,t)-$Lucas
polynomials} to be the sequence of polynomials in variables $q$ and
$t$
\begin{eqnarray}
\label{Lucdeff} L_n(q,t) = \sum_{S \subseteq \{1,2,\dots,n\}~~:~~ S
\cap S_1^{(n)} = \emptyset}
q^{\#\mathrm{~even~elements~in~}S}\hspace{0.7em}t^{\lfloor{n\over
2}\rfloor-\#S}.
\end{eqnarray}
Note that this sum is over subsets $S$ with no two numbers
circularly consecutive.
\end{Def}

These polynomials are a generalization of the sequence of Lucas
polynomials $L_n$ which have the initial conditions $L_1=1$, $L_2=3$
(or $L_0=2$ and $L_1=1$) and satisfy the Fibonacci recurrence $L_n =
L_{n-1} + L_{n-2}$.  The first few Lucas numbers are
$$1,3,4,7,11,18,29,47,76,123,\dots$$
As described in numerous sources, e.g. \cite{SpLuc}, $L_n$ is equal
to the number of ways to color an $n-$beaded necklace black and
white so that no two black beads are consecutive.  You can also
think of this as choosing a subset of $\{1,2,\dots, n\}$ with no
consecutive elements, nor the pair $1,n$.  (We call this circularly
consecutive.) Thus letting $q$ and $t$ both equal one, we get by
definition that $L_n(1,1,)=L_n$.

We prove the following theorem, which relates our newly defined
$(q,t)-$Lucas polynomials to the polynomials of interest, namely the
$N_k$'s.

\begin{Thm}
\label{qtLucEll} \begin{eqnarray} \label{bothsides} 1+q^k-N_k =
L_{2k}(q,-N_1)\end{eqnarray} for all $k\geq 1$.
\end{Thm}

To prove this result it suffices to prove that both sides are equal
for $k \in \{1,2\}$, and that both sides satisfy the same three-term
recurrence relation.  Since
\begin{eqnarray*}
L_2(q,t) &=& 1 + q   + t  \mathrm{~~~~~and}\\
L_4(q,t) &=& 1 + q^2 + (2q+2)t+ t^2
\end{eqnarray*}
we have proven that the initial conditions agree.  Note that the
sets of (\ref{Lucdeff}) yielding the terms of these sums are
respectively
$$\{1\},~\{2\},~\{\hspace{0.5em}\}\mathrm{~~~and~~~~}\{1,3\},~\{2,4\},~\{1\},~\{2\},~\{3\},~\{4\},~\{\hspace{0.5em}\}.$$
It remains to prove that both sides of (\ref{bothsides}) satisfy the
recursion
$$G_{k+1} = (1+q-N_1)G_{k} - qG_{k-1}$$ for $k \geq 1$.

\begin{Prop}
\label{Lrecur} For the $(q,t)-$Lucas polynomials $L_k(q,t)$ defined
as above,
\begin{eqnarray}
\label{l2nrecur} L_{2k+2}(q,t) = (1+q+t)L_{2k}(q,t) -
qL_{2k-2}(q,t). \end{eqnarray}
\end{Prop}

\begin{proof}
To prove this we actually define an auxiliary set of polynomials,
$\{\tilde{L}_{2k}\}$, such that
$$L_{2k}(q,t) = t^{k}\tilde{L}_{2k}(q,t^{-1}).$$
Thus recurrence (\ref{l2nrecur}) for the $L_{2k}$'s translates into
\begin{eqnarray}
\label{l2tnrecur} \tilde{L}_{2k+2}(q,t) =
(1+t+qt)\tilde{L}_{2k}(q,t) - qt^2\tilde{L}_{2k-2}(q,t)
\end{eqnarray}
for the $\tilde{L}_{2k}$'s.  The $\tilde{L}_{2k}$'s happen to have a
nice combinatorial interpretation also, namely

$$\tilde{L}_{2k}(q,t) = \sum_{S \subseteq \{1,2,\dots,2k\}~~:~~ S \cap S_1^{(2k)} = \emptyset}
q^{\#\mathrm{~even~elements~in~}S}~t^{\#S}.$$
Recall our slightly different description which considers these as
the generating function of $2$-colored, labeled necklaces.  We find
this terminology slightly easier to work with.  We can think of the
beads labeled $1$ through $2k+2$ to be constructed from a pair of
necklaces; one of length $2k$ with beads labeled $1$ through $2k$,
and one of length $2$ with beads labeled $2k+1$ and $2k+2$.

Almost all possible necklaces of length $2k+2$ can be decomposed in
such a way since the coloring requirements of the $2k+2$ necklace
are more stringent than those of the pairs.  However not all
necklaces can be decomposed this way, nor can all pairs be pulled
apart and reformed as a $(2k+2)$-necklace.  For example, if $k=2$:

\vspace{2em}

Decomposable \hspace{2.8em} \includegraphics[width = 1.5in, height =
1.5in]{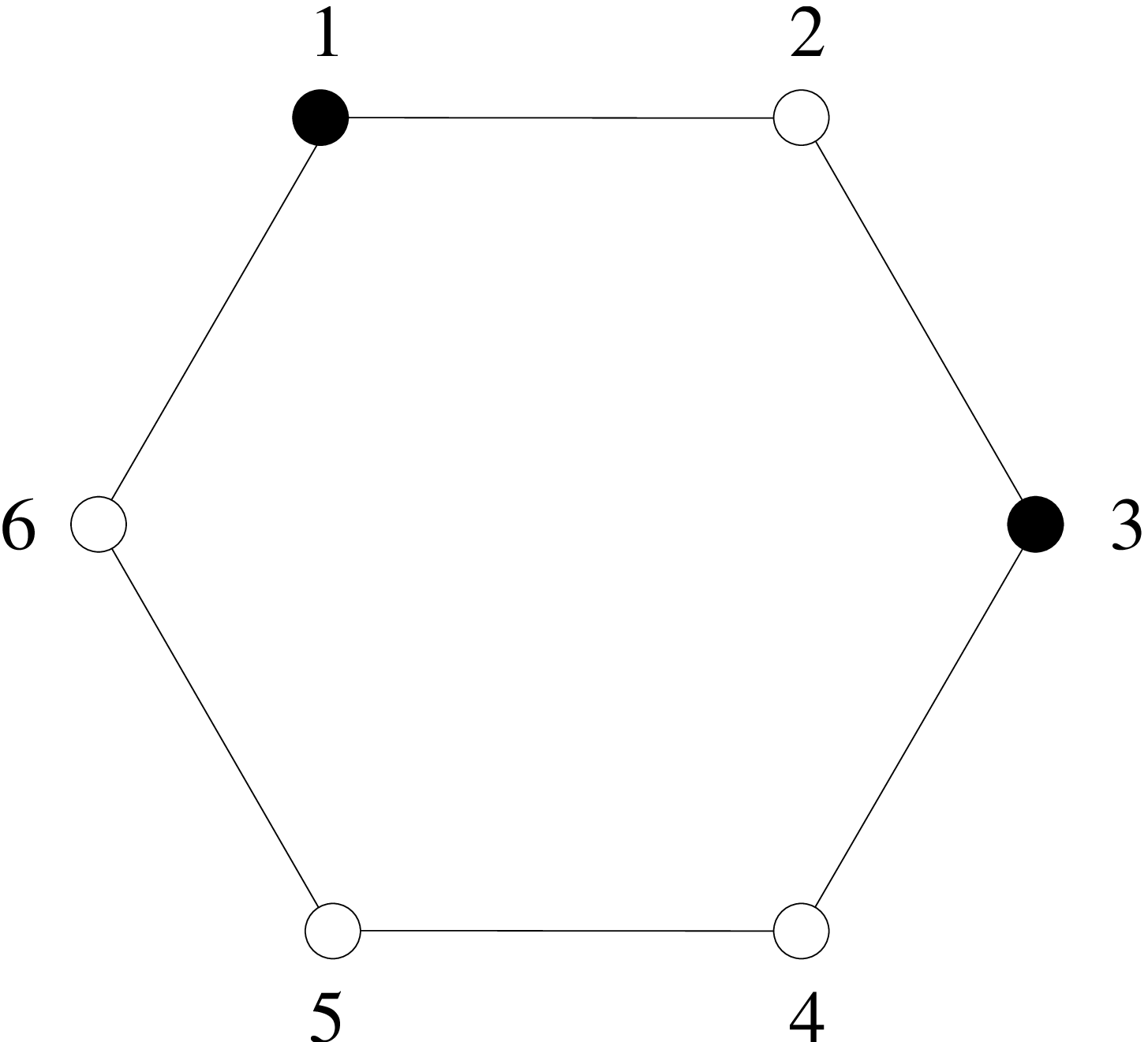} \hspace{1em} $\rightarrow$ \hspace{1em}
\includegraphics[width = 1.5in, height = 1.5in]{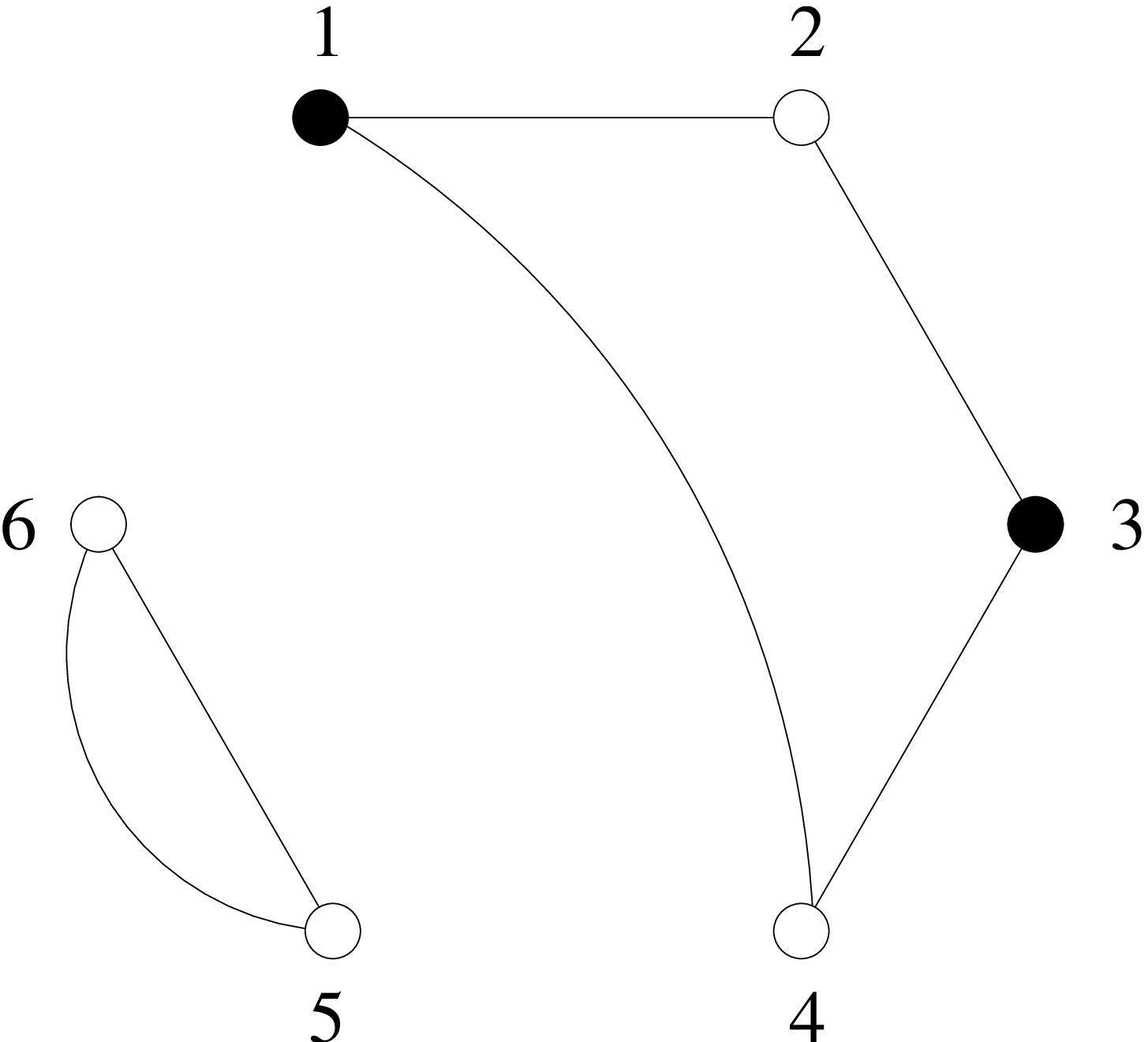}
\vspace{1em}

Not Decomposable \hspace{1em} \includegraphics[width = 1.5in, height
= 1.5in]{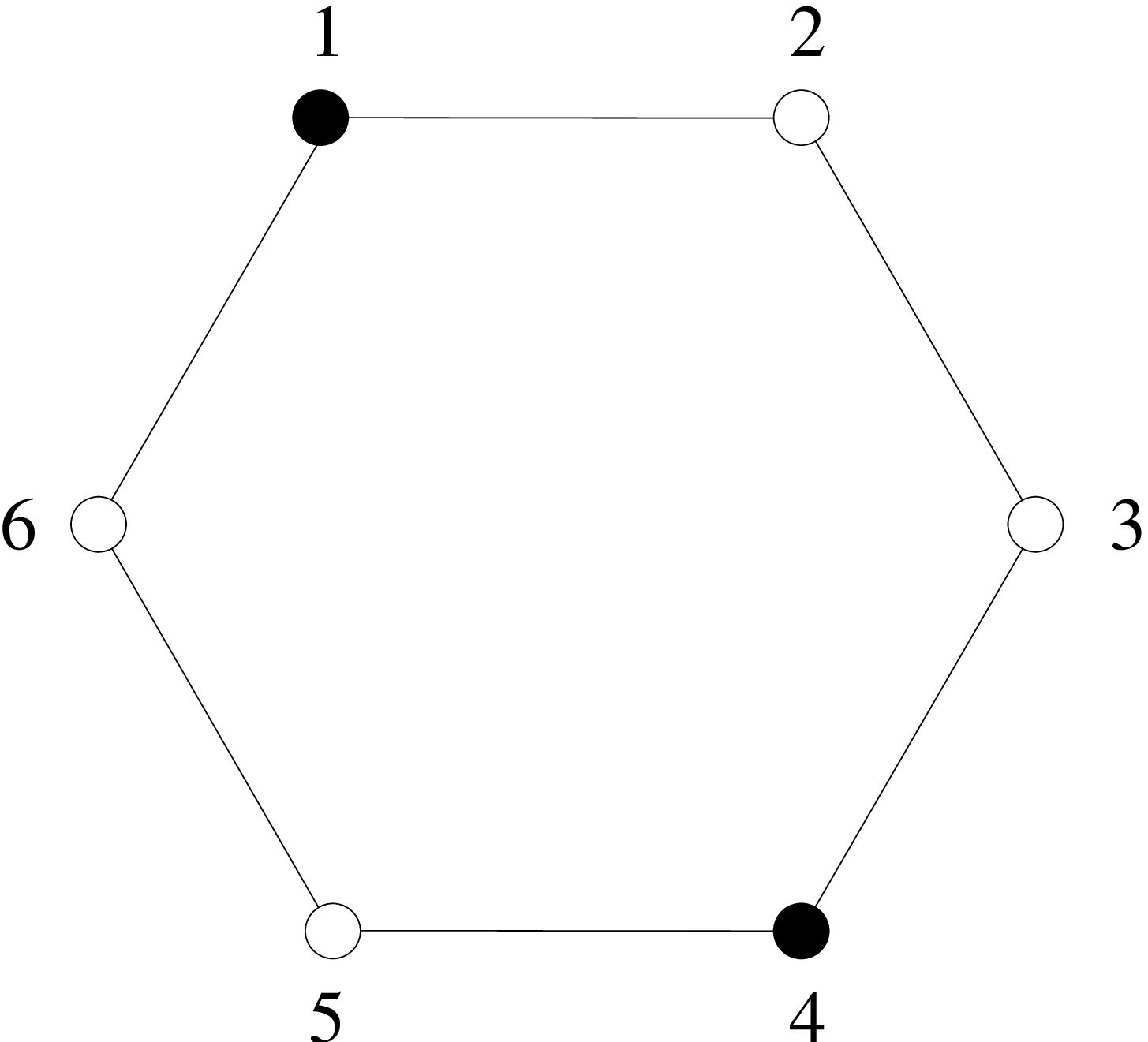} \hspace{1em} $\not \rightarrow$ \hspace{1em}
\includegraphics[width = 1.5in, height = 1.5in]{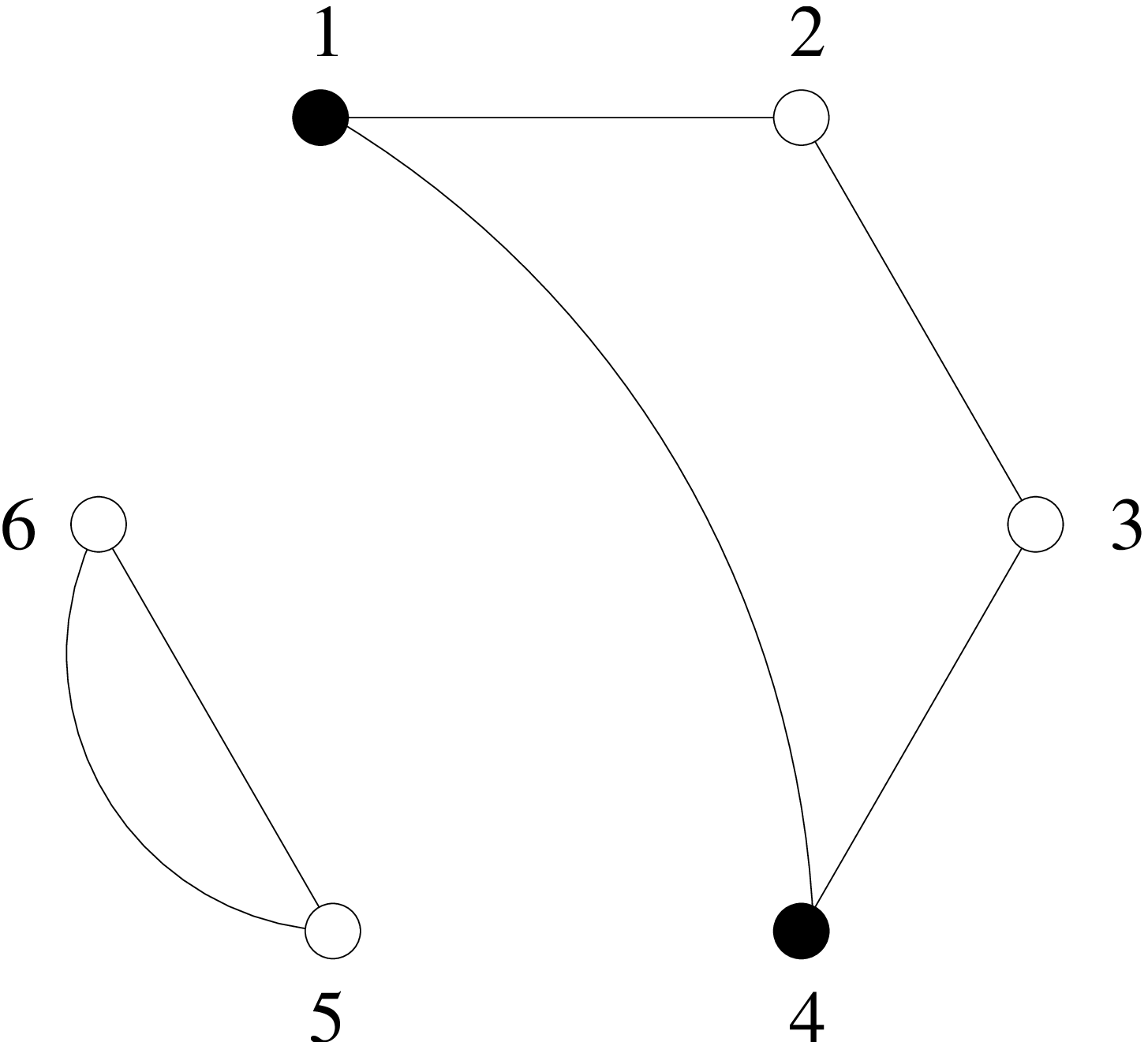}

\vspace{2em}

In these figures, the first necklace is decomposable but the second
one is not since black beads $1$ and $4$ would be adjacent, thus
violating the rule. It is clear enough that the number of pairs is
$\tilde{L}_2(q,t)\tilde{L}_{2k}(q,t) = (1+t+qt)\tilde{L}_{2k}(q,t)$.
To get the third term of the recurrence, i.e.
$qt^2\tilde{L}_{2k-2}$, we must define linear analogues,
$\tilde{F}_{n}(q,t)$'s, of the previous generating function.  Just
as the $\tilde{L}_{n}(1,1)$'s were Lucas numbers, the
$\tilde{F}_{n}(1,1)$'s are Fibonacci numbers.

\begin{Def}
The (twisted) $(q,t)-$Fibonacci polynomials, denoted as
$\tilde{F}_n(q,t)$, are defined as

$$\tilde{F}_k(q,t) = \sum_{S \subseteq \{1,2,\dots,k-1\}~~:~~ S \cap (S_1^{(k-1)} - \{1\}) = \emptyset}
q^{\#\mathrm{~even~elements~in~}S}~t^{\#S}.$$

\end{Def}

The summands here are subsets of $\{1,2,\dots,k-1\}$ such that no
two elements are \emph{linearly} consecutive, i.e. we now allow a
subset with both the first and last elements.  An alternate
description of the objects involved are as (linear) chains of $k-1$
beads which are black or white with no two consecutive black beads.
With these new polynomials at our disposal, we can calculate the
third term of the recurrence, which is the difference between the
number of pairs that cannot be recombined and the number of
necklaces that cannot be decomposed.

\begin{Lem} \label{quant1}
The number of pairs that cannot be recombined into a longer necklace
is $2qt^2\tilde{F}_{2k-2}(q,t).$
\end{Lem}

\begin{proof}
We have two cases: either both $1$ and $2k+2$ are black, or both
$2k$ and $2k+1$ are black.  These contribute a factor of $qt^2$, and
imply that beads $2$, $2k$, and $2k+1$ are white, or that $1$,
$2k-1$, and $2k+2$ are white, respectively.  In either case, we are
left counting chains of length $2k-3$, which have no consecutive
black beads.  In one case we start at an odd-labeled bead and go to
an evenly labeled one, and the other case is the reverse, thus
summing over all possibilities yields the same generating function
in both cases.
\end{proof}

\begin{Lem} \label{quant2}
The number of $(2k+2)$-necklaces that cannot be decomposed into a
$2$-necklace and a $2k$-necklace is $qt^2\tilde{F}_{2k-3}(q,t)$.
\end{Lem}

\begin{proof}
The only ones that cannot be decomposed are those which have beads
$1$ and $2k$ both black. Since such a necklace would have no
consecutive black beads, this implies that beads $2$, $2k-1$,
$2k+1$, and $2k+2$ are all white.  Thus we are reduced to looking at
chains of length $2k-4$, starting at an odd, $3$, which have no
consecutive black beads.
\end{proof}

\begin{Lem} The difference of the quantity referred to in Lemma \ref{quant2} from the quantity in Lemma \ref{quant1}
is exactly $qt^2\tilde{L}_{2k-2}(q,t)$.
\end{Lem}

\begin{proof}
It suffices to prove the relation
$$qt^2\tilde{L}_{2k-2}(q,t) =
2qt^2\tilde{F}_{2k-2}(q,t) -qt^2\tilde{F}_{2k-3}(q,t)$$ which is
equivalent to
\begin{eqnarray}
\label{Lucide} qt^2\tilde{L}_{2k-2}(q,t) =
qt^2\tilde{F}_{2k-2}(q,t)+q^2t^3\tilde{F}_{2k-4}(q,t) \end{eqnarray}
since
\begin{eqnarray} \label{FibId}
\tilde{F}_{2k-2}(q,t) = qt\tilde{F}_{2k-4}(q,t) +
\tilde{F}_{2k-3}(q,t).
\end{eqnarray}
Note that identity (\ref{FibId}) simply comes from the fact that the
$(2k-2)$nd bead can be black or white.  Finally we prove
(\ref{Lucide}) by dividing by $qt^2$, and then breaking it into the
cases where bead $1$ is white or black.  If bead $1$ is white, we
remove that bead and cut the necklace accordingly.  If bead $1$ is
black, then beads $2$ and $2k+2$ must be white, and we remove all
three of the beads.
\end{proof}

With this lemma proven, the recursion for the $\tilde{L}_{2k}$'s, hence the
$L_{2k}$'s follows immediately.
\end{proof}

\begin{Prop} \label{akId}
For an elliptic curve $C$ with $N_k$ points over $\f_{q^k}$ we have
that $$1+q^{k+1} - N_{k+1} = (1+q-N_1)(1+q^k-N_k) -
q(1+q^{k-1}-N_{k-1}).$$
\end{Prop}

\begin{proof}
Recalling that for an elliptic curve $C$ we have the identity
$$N_k = 1 + q^k - \alpha_1^k -\alpha_2^k,$$ we can rewrite the
statement of this proposition as
\begin{eqnarray}
\label{alpheqeasy} \alpha_1^{k+1}+\alpha_2^{k+1} =
(\alpha_1+\alpha_2)(\alpha_1^k+\alpha_2^k) - q(\alpha_1^{k-1}+\alpha_2^{k-1}).
\end{eqnarray} Noting that $q=\alpha_1\alpha_2$ we obtain this proposition after
expanding out algebraically the right-hand-side of (\ref{alpheqeasy}).
\end{proof}

With the proof of Propositions \ref{Lrecur} and \ref{akId}, we have
proven Theorem \ref{qtLucEll}.

\subsection{$(q,t)-$Wheel polynomials} \label{whnums}

Given that the Lucas numbers are related to the polynomial formulas
$N_k(q,N_1)$, a natural question concerns how alternative
interpretations of the Lucas numbers can help us better understand
$N_k$.  As noted in \cite{SpLuc}, \cite{firsttree}, and \cite[Seq.
A004146]{intseq}, the sequence $\{L_{2n}-2\}$ counts the number of
spanning trees in the wheel graph $W_n$; a graph which consists of
$n+1$ vertices, $n$ of which lie on a circle and one vertex in the
center, a hub, which is connected to all the other vertices.

We note that a spanning tree $T$ of $W_n$ consists of spokes and a
collection of disconnected arcs on the rim. Further, since there are
no cycles and $T$ is connected, each spoke intersects exactly one
arc. (Since it will turn out to be convenient in the subsequent
considerations, we make the -- somewhat counter-intuitive --
convention that an isolated vertex is considered to be an arc of
length $1$, and more generally, an arc consisting of $k$
\emph{vertices} is considered as
an arc of \emph{length} $k$.) %
%
We imagine the circle being oriented clockwise, and imagine the tail
of each arc being the vertex which is the sink for that arc. In the
case of an isolated vertex, the lone vertex is the tail of that arc.
Since the spoke intersects each arc exactly once, if an arc has
length $k$, meaning that it contains $k$ vertices, there are $k$
choices of where the spoke and the arc meet.  We define the
$q-$weight of an arc to be
$q^{\mathrm{~number~of~edges~between~the~spoke~and~the~tail}}$,
abbreviating this exponent as $spoke-tail$ distance. We define the
$q-$weight of the tree to be the product of the $q-$weights for all
arcs on the rim of the tree. This combinatorial interpretation
motivates the following definition.

\begin{Def}
$$\mathcal{W}_n(q,t) = \sum_{T\mathrm{~a~spanning~tree~of~}W_n}
q^{\mathrm{sum~of~spoke-tail~distance~in~}T}~t^{\#~\mathrm{spokes~of~}T}.$$

Here the exponent of $t$ counts the number of edges emanating from
the central vertex, and the exponent of $q$ is as above.
\end{Def}

\vspace{2em}

$q^2t^3$ \hspace{1em} \includegraphics[width = 1.5in, height =
1.5in]{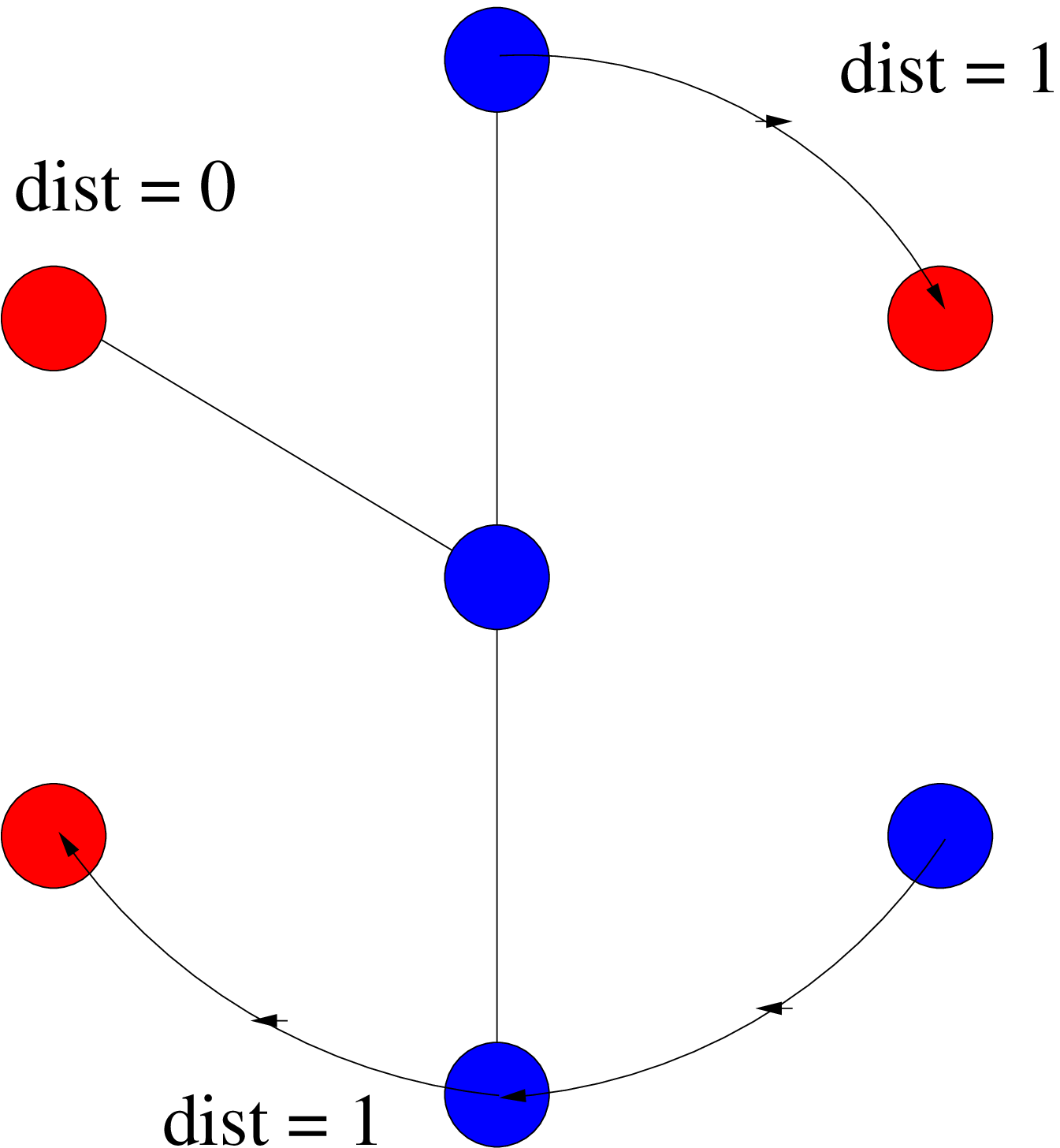} \hspace{1em} $q^3t^3$ \hspace{1em}
\includegraphics[width = 1.5in, height = 1.5in]{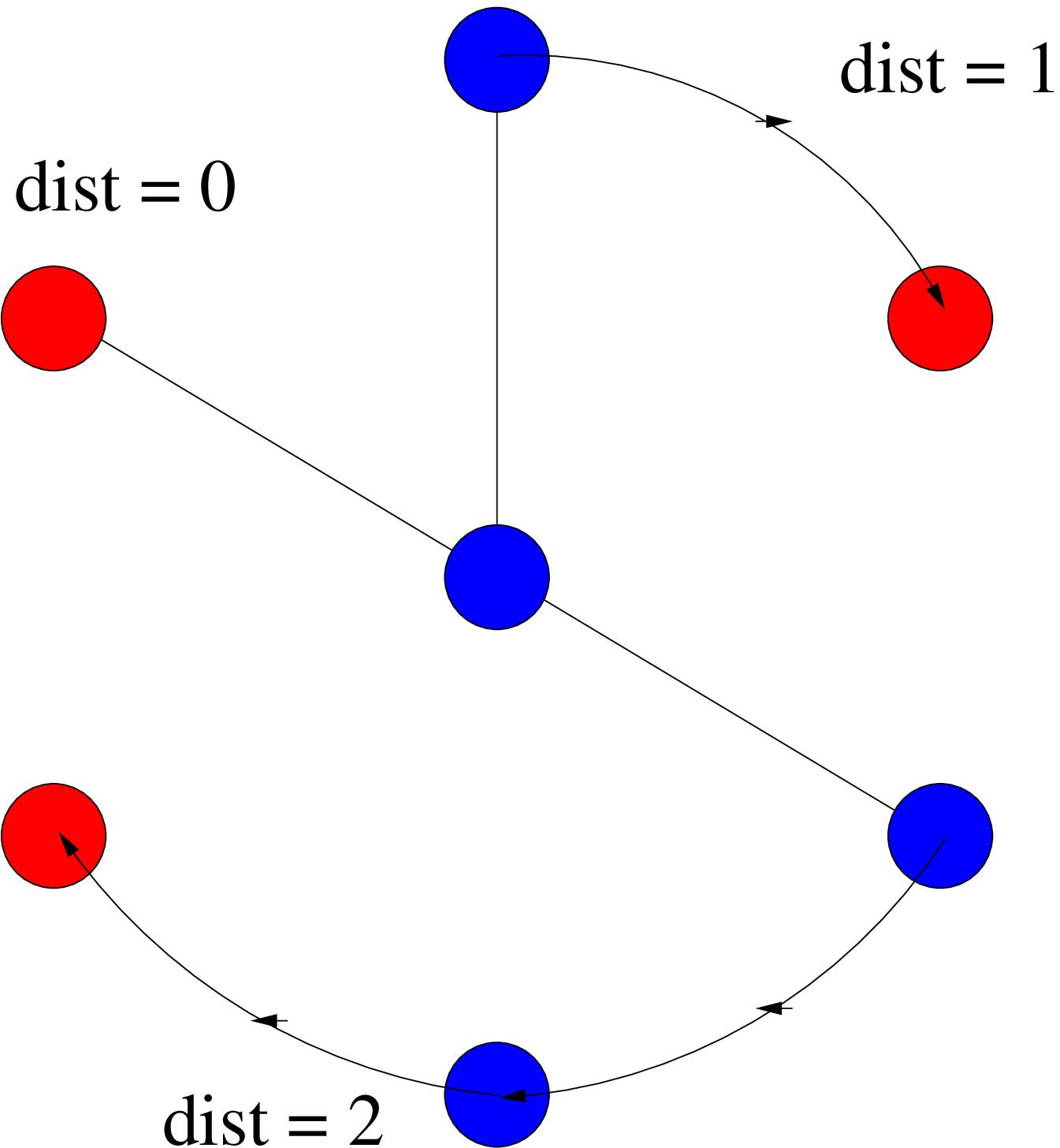}

\vspace{2em}

\noindent This definition actually provides exactly the generating
function that we desired.

\begin{Thm}
\label{whser}
$$N_k = -\mathcal{W}_k(q,-N_1)$$ for all $k\geq 1$.
\end{Thm}

Notice that this yields an exact interpretation of the $P_{i,k}$
polynomials as follows:

$$P_{i,k}(q) = \sum_{T\mathrm{~a~spanning~tree~of~}W_n\mathrm{~with~exactly~}i\mathrm{~spokes}}
q^{\mathrm{sum~of~spoke-tail~distance~in~}T}.$$

We prove this theorem in two different ways.  The first method
utilizes Theorem \ref{qtLucEll} and an analogue of the bijection
given in \cite{SpLuc} which relates perfect and imperfect matchings
of the circle of length $2k$ and spanning trees of $W_k$. Our second
proof uses the observation that we can categorize the spanning trees
based on the sizes of the various connected arcs on the rims. Since
this categorization corresponds to partitions, this method exploits
formulas for decomposing power symmetric function $p_k$ into a
linear combination of $h_{\lambda}$'s, as described in Section
\ref{circtabloid}.

\subsection{First proof of Theorem \ref{whser}: Bijective}
\label{firpfwhser}

There is a simple bijection between subsets of $\{1,2,\dots, 2n\}$
with size at most $n-1$ as well as no two elements circularly
consecutive and spanning trees of the wheel graph $W_n$. We use this
bijection to give our first proof of Theorem \ref{whser}. The
bijection is as follows:

Given a subset $S$ of the set $\{1,2,\dots, 2n-1,2n\}$ with no
circularly consecutive elements, we define the corresponding
spanning tree $T_S$ of $W_n$ (with the correct $q$ and $t$ weight)
in the following way:

1) We use the convention that the vertices of the graph $W_n$ are
labeled so that the vertices on the rim are $w_1$ through $w_n$, and
the central vertex is $w_0$.

2) We exclude the two subsets which consist of all the odds or all
the evens from this bijection.  Thus we only look at subsets which
contain $n-1$ or fewer elements.

3) For $1 \leq i \leq n$, an edge exists from $w_0$ to $w_i$ if and
only if neither $2i-2$ nor $2i-1$ (element $0$ is identified with
element $2n$) is contained in $S$.

4) For $1 \leq i \leq n$, an edge exists from $w_i$ to $w_{i+1}$
($w_{n+1}$ is identified with $w_1$) if and only if element $2i-1$
or element $2i$ is contained in $S$.

\vspace{2em}

\includegraphics[width = 2in, height =
2in]{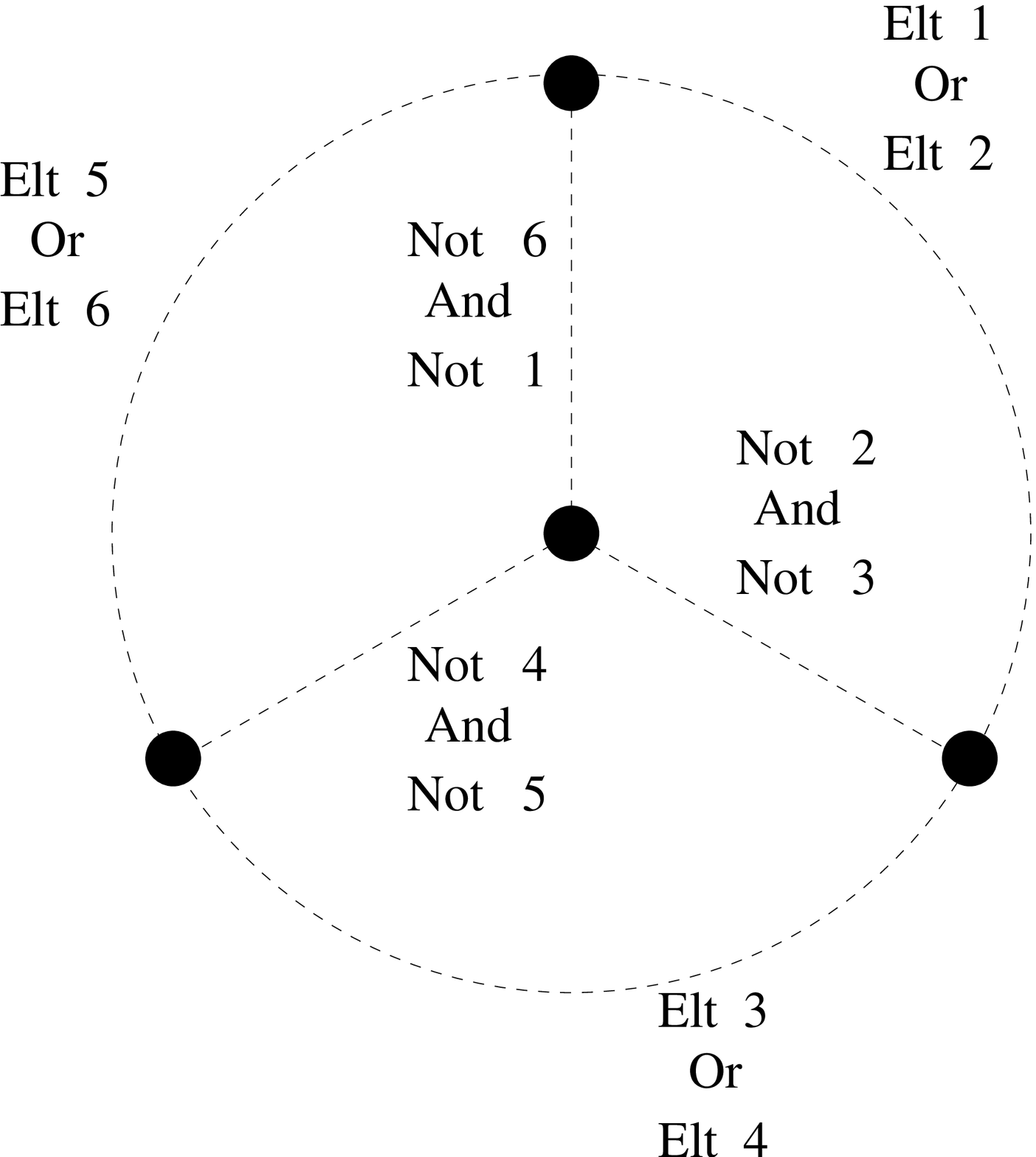} \hspace{2em}  $\bigg\{\hspace{0.5em}\bigg\}
\longleftrightarrow$ \hspace{1em}
\includegraphics[width = 1.8in, height = 1.3in]{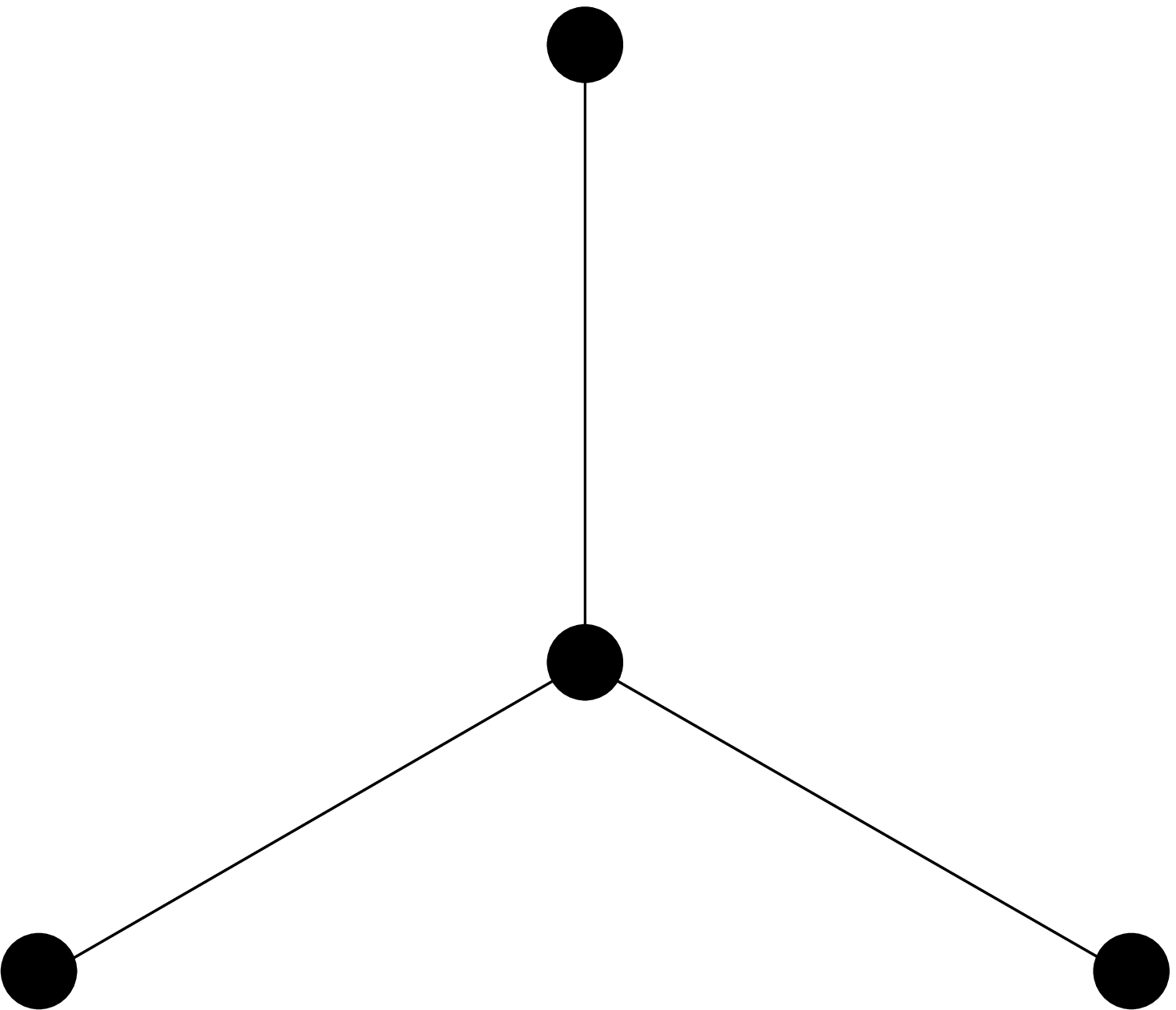}

$\bigg\{3\bigg\} \longleftrightarrow$
\includegraphics[width = 1.3in, height = 1.3in]{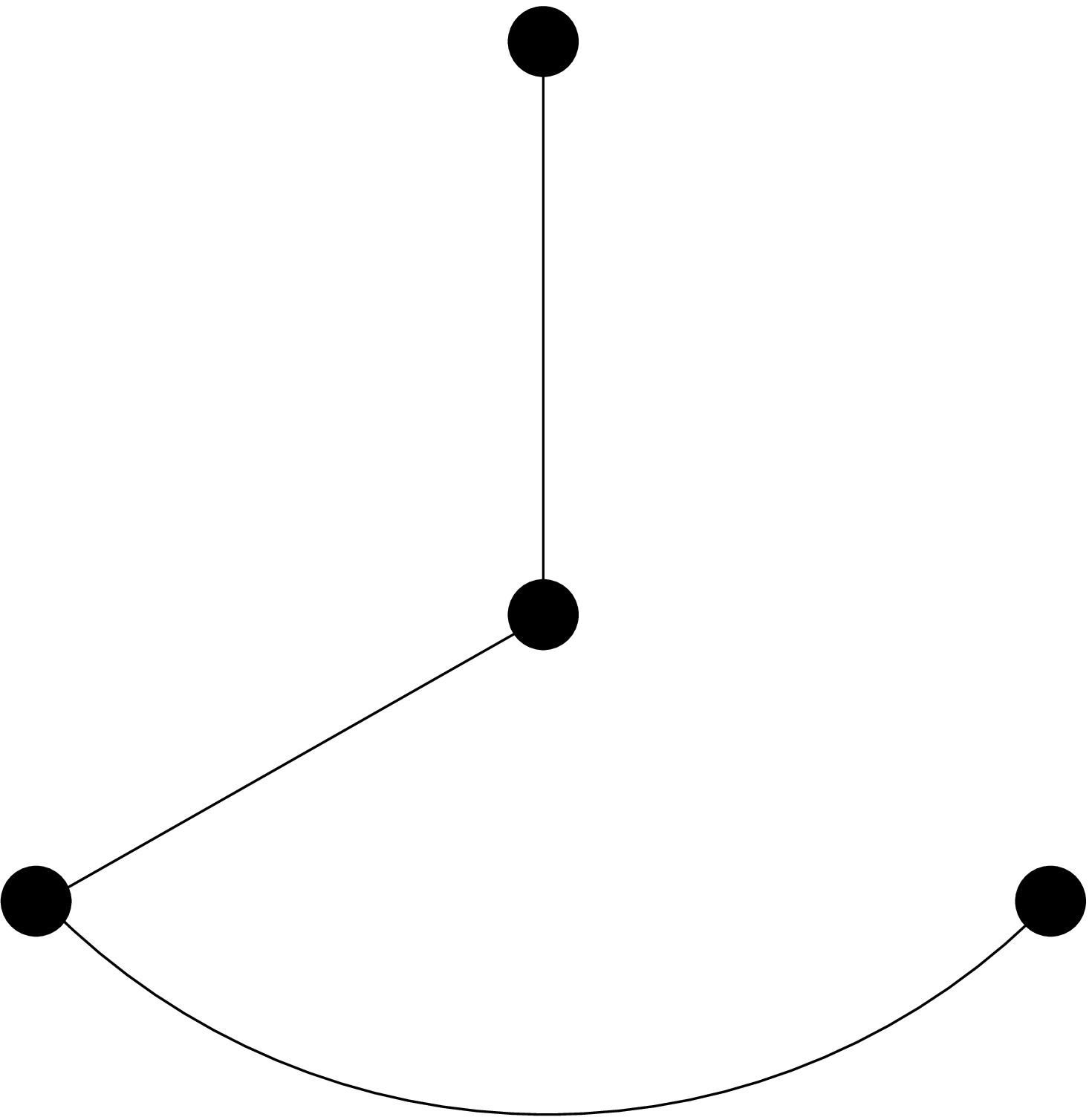} \hspace{2em}
 $\bigg\{2,5\bigg\} \longleftrightarrow$
\includegraphics[width = 1.8in, height = 1.3in]{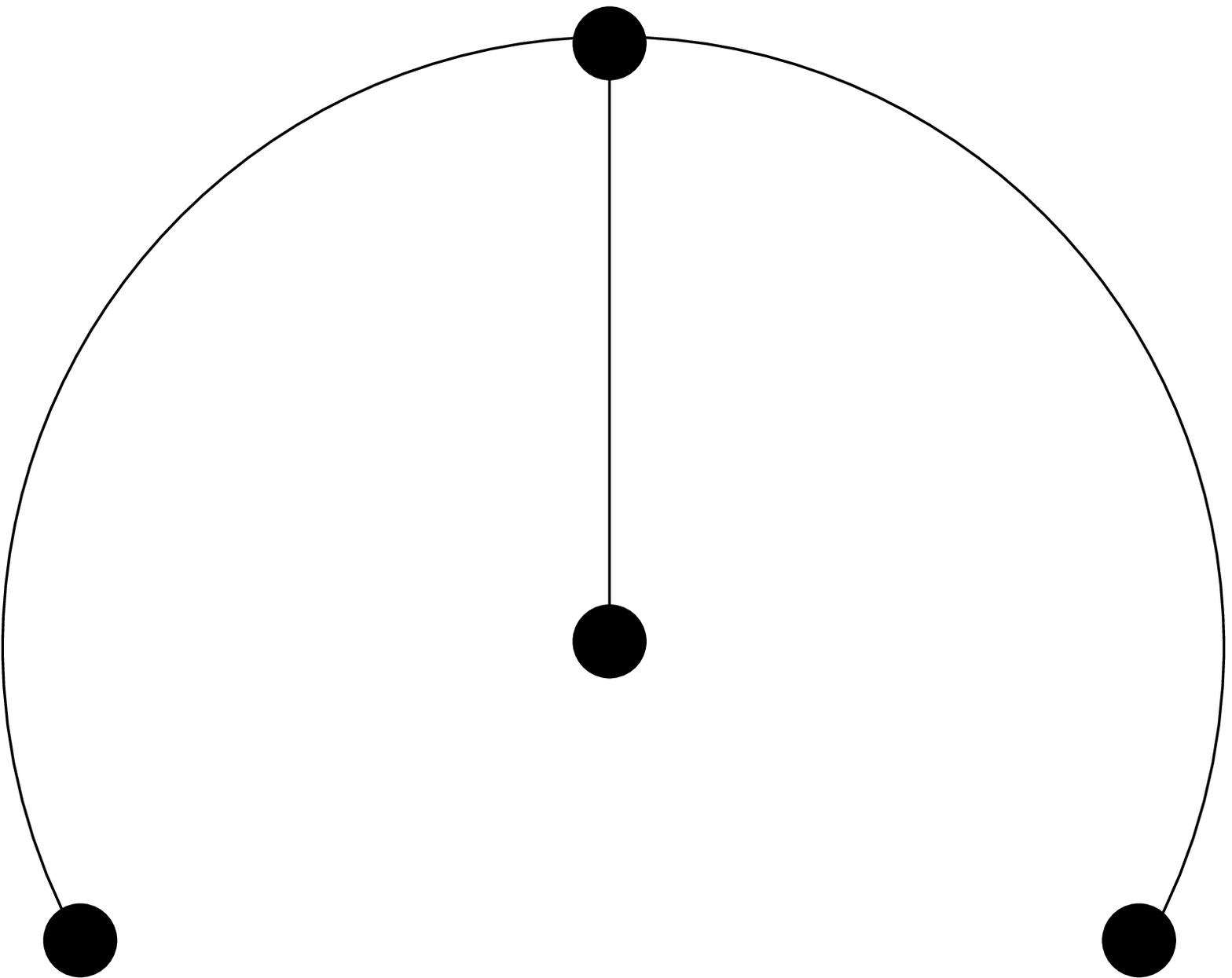}

\vspace{2em}

\begin{Prop}
Given this construction, $T_S$ is in fact a spanning tree of $W_n$
and further, tree $T_S$ has the same $q-$weights and $t-$weights as
set $S$.
\end{Prop}

\begin{proof}
Suppose that set $S$ contains $k$ elements.  From our above
restriction, we have that $0 \leq k \leq n-1$. Since $S$ is a
$k$-subset of a $2n$ element set with no circularly consecutive
elements, there are $(n-k)
\mathrm{~pairs~}\{2i-2,2i-1\}\mathrm{~with~neither~element~in~set~}S$,
$\mathrm{and~}k
\mathrm{~pairs~}\{2i-1,2i\}\mathrm{~with~one~element~in~set~}S$.
Consequently, subgraph $T_S$ consists of exactly $(n-k)+k=n$ edges.
Since $n=(\#~$vertices of $W_n)-1$, to prove $T_S$ is a spanning
tree, it suffices to show that each vertex of $W_n$ is included. For
every oddly-labeled element of $\{1,2,\dots, 2n\}$, i.e. $2i-1$ for
$1\leq i \leq n$, we have the following rubric:

1) If $(2i-1)\in S$ then the subgraph $T_S$ contains the edge from
$w_i$ to $w_{i+1}$.

2) If $(2i-1)\not\in S$ and additionally $(2i-2)\not\in S$, then
$T_S$ contains the spoke from $w_0$ to $w_i$.

3) If $(2i-1)\not\in S$ and additionally $(2i-2)\in S$, then $T_S$
contains the edge from $w_{i-1}$ to $w_i$.

\noindent Since one of these three cases happens for all $1\leq
i\leq n$, vertex $w_i$ is incident to an edge in $T_S$.  Also, the
central vertex, $w_0$, has to be included since by our restriction,
$0 \leq k \leq n-1$, there are $(n-k) \geq 1$ pairs $\{2i-2,2i-1\}$
which contain no elements of $S$.

The number of spokes in $T_S$ is $(n-k)$ which agrees with the
$t-$weight of a set $S$ with $k$ elements.  Finally, we prove that
the $q$-weight is preserved, by induction on the number of elements
in the set $S$.  If set $S$ has no elements, the $q-$weight should
be $q^0$, and spanning tree $T_S$ will consist of $n$ spokes which
also has $q-$weight $q^0$.

Now given a $k$ element subset $S$ ($0\leq k\leq n-2$), it is only
possible to adjoin an odd number if there is a sequence of three
consecutive numbers starting with an even, i.e. $\{2i-2,2i-1,2i\}$,
which is disjoint from $S$.  Such a sequence of $S$ corresponds to a
segment of $T_S$ where a spoke and tail of an arc intersect.  (Note
this includes the case of vertex $w_i$ being an isolated vertex.)

In this case, subset $S^\prime = S \cup \{2i-1\}$ corresponds to
$T_{S^\prime}$, which is equivalent to spanning tree $T_S$ except
that one of the spokes $w_0$ to $w_i$ has been deleted and replaced
with an edge from $w_i$ to $w_{i+1}$.  The arc corresponding to the
spoke from $w_i$ will now be connected to the next arc, clockwise.
Thus the distance between the spoke and the tail of this arc will
not have changed, hence the $q-$weight of $T_{S^\prime}$ will be the
same as the $q-$weight of $T_{S}$.

Alternatively, it is only possible to adjoin an even number to $S$
if there is a sequence $\{2i-1,2i,2i+1\}$ which is disjoint from
$S$. Such a sequence of $S$ corresponds to a segment of $T_S$ where
a spoke meets the \emph{end} of an arc.  (Note this includes the
case of vertex $w_i$ being an isolated vertex.)

Here, subset $S^{\prime\prime} = S \cup \{2i\}$ corresponds to
$T_{S^{\prime\prime}}$, which is equivalent to spanning tree $T_S$
except that one of the spokes $w_0$ to $w_{i+1}$ has been deleted
and replaced with an edge from $w_i$ to $w_{i+1}$. The arc
corresponding to the spoke from $w_{i+1}$ will now be connected to
the \emph{previous} arc, clockwise. Thus the cumulative change to
the total distance between spokes and the tails of arcs will be an
increase of one, hence the $q-$weight of $T_{S^{\prime\prime}}$ will
be $q^1$ times the $q-$weight of $T_{S}$.

Since any subset $S$ can be built up this way from the empty set,
our proof is complete via this induction.
\end{proof}

Since the two sets we excluded, of size $k$ had $(q,t)-$weights
$q^0t^0$ and $q^kt^0$ respectively, we have proven Theorem
\ref{whser}.

\subsection{Second proof of Theorem \ref{whser}: Via generating
function identities} \label{circtabloid}

For our second proof of Theorem \ref{whser}, we consider writing the
zeta function as an ordinary generating function instead, i.e.
\begin{eqnarray} \label{zetaAsH} Z(C,T) = 1+ \sum_{k\geq 1}H_k T^k.\end{eqnarray}  In such a form, the
$H_k$'s are positive integers which enumerate the number of
effective $C(\f_q)$-divisors of degree $k$, as noted in several
places, such as \cite{Mor}.

\begin{Prop} \label{Nsy}
\begin{eqnarray} \label{Npowhog} N_k = \sum_{\lambda \vdash k} (-1)^{l(\lambda)-1}{k \over
l(\lambda)}{l(\lambda) \choose d_1,~d_2,~\dots~d_m}
\prod_{i=1}^{l(\lambda)} H_{\lambda_i}.\end{eqnarray}
\end{Prop}

\begin{proof}

Comparing formulas (\ref{zetaF}) and (\ref{zetaAsH}) for $Z(C,T)$
and taking logarithms, we obtain
\begin{eqnarray*}
{N_k \over k} &=& \log Z(C,T)\bigg|_{T^k} = \log \bigg(1+\sum_{n\geq
1} H_n T^n\bigg)\bigg|_{T^k} = \sum_{m\geq 1} {(-1)^{m-1}
\bigg(\sum_{n= 1}^k H_n T^n\bigg)^m \over m}\bigg|_{T^k}.
\end{eqnarray*}
To obtain the coefficient of $T^k$ in \begin{eqnarray}
\label{insideprod} \bigg(H_1T+H_2T^2+\dots +H_k
T^k\bigg)^m,\end{eqnarray} we first select a partition of $k$ with
length $\ell(\lambda)=m$.  In other words, $\lambda$ is a vector of
positive integers satisfying $\lambda_1\geq \lambda_2\geq \dots \geq
\lambda_m$.  Each occurrence of $\lambda_i=j$ in this partition
corresponds to choosing summand $H_jT^j$ in the $i$th term in
product (\ref{insideprod}).  Secondly, since the order of these
terms does not matter, we include multinomial coefficients. Finally,
multiplying through by $k$ yields formula (\ref{Npowhog}) for $N_k$.
\end{proof}

\begin{Rem}
The same manipulations done above for the generating functions are
analogous to identities which relate the power symmetric functions
and homogeneous symmetric functions.  See for example \cite{brick},
\cite{MacDon}, or \cite[pg. 21]{EC2}.  This is no coincidence, and
in particular the terminology of plethysm provides a rigorous
connection between symmetric functions and the enumeration of points
on curves. See Section \ref{bifibsec} below, \cite{Gars}, or
\cite{MusThesis} for more details on plethysm and this connection.
\end{Rem}

\begin{Rem}
The above algebraic reasoning can also be translated into a
combinatorial description of how points on $C$ over $\f_{q^k}$ can
be enumerated using inclusion-exclusion, and points over smaller
extension fields. See \cite[Chap 4.]{MusThesis} for more details.
\end{Rem}

We now specialize to the case of $g=1$.  Here we can write $H_k$ in
terms of $N_1$ and $q$.  We expand the series \begin{eqnarray}
\label{Zsimp}Z(E,T) = {1-(1+q-N_1)T + qT^2 \over (1-T)(1-qT)} = 1 +
{N_1T \over (1-T)(1-qT)}\end{eqnarray} with respect to $T$, and
obtain $H_0 = 1$ and $H_k = N_1(1+q+q^2+\dots+q^{k-1})$ for $k\geq
1$. Plugging these into formula (\ref{Npowhog}), we get polynomial
formulas for $N_k$ in terms of $q$ and $N_1$
$$N_k = \sum_{\lambda \vdash k} (-1)^{l(\lambda)-1}{k \over l(\lambda)}{l(\lambda)
\choose d_1,~d_2,~\dots~d_k} \bigg(\prod_{i=1}^{l(\lambda)}
(1+q+q^2+\dots + q^{\lambda_i-1})\bigg)N_1^{l(\lambda)}.$$
Consequently, Theorem \ref{whser} is true if and only if we can
replace $N_1$ with $-t$ and then multiply by $(-1)$ and get a true
expression for $\mathcal{W}_k$, the $(q,t)$-weighted number of
spanning trees on the wheel graph $W_k$.
We thus provide the following combinatorial argument for the
required formula.

\begin{Prop} \label{Wsy}
\begin{eqnarray} \label{Wform} ~~~~~\mathcal{W}_k = \sum_{\lambda \vdash k} {k \over
l(\lambda)}{l(\lambda) \choose d_1,~d_2,~\dots~d_k}
\bigg(\prod_{i=1}^{l(\lambda)} (1+q+q^2+\dots +
q^{\lambda_i-1})\bigg)t^{l(\lambda)}.\end{eqnarray}
\end{Prop}
\begin{proof}
We construct a spanning tree of $W_k$ from the following choices:
First we choose a partition $\lambda = 1^{d_1}2^{d_2}\cdots k^{d_m}$
of $k$.  We let this dictate how many arcs of each length occur,
i.e. we have $d_1$ isolated vertices, $d_2$ arcs of length $2$, etc.
Note that this choice also dictates the number of spokes, which is
equal to the number of arcs, i.e. the length of the partition.

Second, we pick an arrangement of the $l(\lambda)$ arcs on the
circle. After picking one arc to start with, without loss of
generality since we are on a circle, we have $${1 \over l(\lambda)}
{l(\lambda) \choose d_1,~d_2,~\dots~d_m}$$ choices for such an
arrangement.
Third, we pick which vertex $w_i$ of the rim to start with. There
are $k$ such choices.
Fourth, we pick where the $l(\lambda)$ spokes actually intersect the
arcs.  There are $|$arc$|$ choices for each arc, and the $q-$weight
of this sum is $(1+q+q^2+\dots+q^{|\mathrm{arc}|})$ for each arc.
Summing up all the possibilities yields (\ref{Wform}) as desired.
\end{proof}

Thus we have given a second proof of Theorem \ref{whser}.

\section{More on bivariate Fibonacci polynomials via duality}
\label{bifibsec}

In this section we explore further properties of various sequences
of coefficients arising from the zeta function of a curve, and also
more properties regarding bivariate Fibonacci polynomials.  Our
tools for such investigations consists of two different
manifestations of duality.

\subsection{Duality between the symmetric functions $h_k$ and $e_k$}

Given the usefulness of symmetric functions in discovering the
identities described by Propositions \ref{Nsy} and \ref{Wsy}, we now
illustrate further applications of the plethystic view of the zeta
function.

The symmetric functions that we utilize in this paper are the
power symmetric functions $p_k$, the complete homogeneous symmetric
functions $h_k$, and the elementary symmetric functions $e_k$. Given
the alphabet $\{x_1,x_2,\dots, x_n\}$, each of these
can be written as \begin{eqnarray*} p_k &=& x_1^k + x_2^k + \dots + x_n^k, \\
h_k &=& \sum_{\stackrel{0 \leq i_1, i_2, \dots, i_n \leq
k}{i_1+i_2+\dots + i_n
= k}} x_1^{i_1}x_2^{i_2}\cdots x_n^{i_n}, \mathrm{~~and} \\
e_k &=& \sum_{1 \leq i_1 < i_2 < \dots < i_k \leq n}
x_{i_1}x_{i_2}\cdots x_{i_k}.
\end{eqnarray*}
In general, a plethystic substitution of a formal power series
$F(t_1,t_2,\dots)$ into a symmetric polynomial $A(x)$, denoted as
$A[E]$, is obtained by setting $$A[E] =
Q_A(p_1,p_2,\dots)|_{p_k\rightarrow E(t_1^k,t_2^k,\dots)},$$ where
$Q_A(p_1,p_2,\dots)$ gives the expansion of $A$ in terms of the
power sums basis $\{p_\alpha\}_\alpha$.  The main example of this
technique that we use is $N_k = p_k[1+q-\alpha_1-\alpha_2-\dots -
\alpha_{2g}]$ for a genus $g$ curve.

To begin, we use the following well-known symmetric function
identity

\begin{eqnarray*}
\prod_{k\in \mathcal{I}} {1 \over 1-
t_kT} &=& \exp\bigg(\sum_{n\geq 1} p_n {T^n \over n}  \bigg) \\
&=& \sum_{n\geq 0} h_nT^n \\
&=& {1 \over \sum_{n \geq 0} (-1)^n e_nT^n}
\end{eqnarray*}
where $h_n$, $p_n$, and $e_n$ are symmetric functions in the
variables $\{t_k\}_{k\in\mathcal{I}}$. \cite[pgs. 21, 296]{EC2}
The zeta function $Z(C,T)$ is equal to all of these for a certain
choice of $\{t_k\}_{k\in \mathcal{I}}$ and consequently, we get that
\begin{eqnarray}\label{ZasE} Z(C,T) = {1 \over \sum_{k\geq 0} (-1)^k E_k \cdot T^k}\end{eqnarray}
where $E_k = e_k[1+q-\alpha_1-\alpha_2-\dots - \alpha_{2g}]$.

\begin{Rem}
Like the $N_k$'s and $H_k$'s, the $E_k$'s also have an algebraic
geometric interpretation, namely $E_k$ equals the \emph{signed}
number of positive divisors $D$ of degree $k$ on curve $C$
\emph{such that no prime divisor appears more than once in} $D$.
This follows from the reciprocity between $h_k$ and $e_k$ which is
analogous to the reciprocity between \emph{choose} and
\emph{multi-choose}, i.e. choice with replacement.
\end{Rem}

Recall that in Section \ref{LucSec}, we defined $\tilde{F}_k(q,t)$,
i.e. the twisted $(q,t)$-Fibonacci polynomials.  Here we define
$F_k(q,t)$, an alternative bivariate analogue of the Fibonacci
numbers. The definition of $F_k(q,t)$ is identical to that of
$\tilde{F}_k(q,t)$ except for the weighting of parameter $t$.

\begin{Def}
We define the $(q,t)$-{\bf Fibonacci polynomials} to be the sequence
of polynomials in variables $q$ and $t$ given by
$${F}_k(q,t) = \sum_{S \subseteq \{1,2,\dots,k-1\}~~:~~ S \cap (S_1^{(k-1)} - \{1\}) = \emptyset}
q^{\#\mathrm{~even~elements~in~}S}~t^{\lceil {k\over 2}\rceil-\#S}.$$
\end{Def}

From this definition we obtain the following formulas for the
$E_k$'s in the elliptic case.
\begin{Thm} \label{EkasFib}
If $C$ is a genus one curve, and the $E_k$'s are as above, then for
$n\geq 1$, $E_{-n}=0$, $E_0=1$, and
$$E_n = (-1)^n F_{2n-1}(q,-N_1)$$ where $E_k$ and $F_k(q,t)$ are as
defined above.
\end{Thm}

\noindent The expansions for the first several $E_k$'s, i.e. $F_{2k-1}(q,t)$'s,
are given below.
\begin{eqnarray*}
E_1 &=& N_1 \\
E_2 &=& -(1+q)N_1 + N_1^2 \\
E_3 &=& (1+q+q^2)N_1 - (2+2q)N_1^2 + N_1^3 \\
E_4 &=& -(1+q+q^2+q^3)N_1 + (3+4q+3q^2)N_1^2 - (3+3q)N_1^3 + N_1^4 \\
E_5 &=& (1+q+q^2+q^3+q^4)N_1 - (4+6q+6q^2+4q^3)N_1^2 + (6+9q+6q^2)N_1^3 -
(4+4q)N_1^4 + N_1^5
\end{eqnarray*}

Before proving Theorem \ref{EkasFib} we develop two key propositions.
\begin{Prop} \label{FibRecurEk}
$F_{2n+1}(q,t) = (1+q+t)F_{2n-1}(q,t) - qF_{2n-3}(q,t)$ for $n\geq 2$.
\end{Prop}

\begin{proof}
This follows the similar logic as the proof of Proposition
\ref{Lrecur} except we can use a more direct method.  (One can use
the $t$-weighting of the twisted $(q,t)$-Fibonacci polynomials
instead to see this recursion more clearly, but we omit this
detour.)  The polynomial $F_{2n+1}$ is a $(q,t)$-enumeration of the
number of chains of $2n$ beads, with each bead either black or
white, and no two consecutive beads both black.  Similarly
$(1+q+t)F_{2n-1}$ enumerates the concatenation of such a chain of
length $2n-2$ with a chain of length $2$.  One can recover a legal
chain of length $2n$ this way except in the case where the
$(2n-2)$nd and $(2n-1)$st beads are both black.  Such cases are
enumerated by $qF_{2n-3}$ and this completes the proof.
\end{proof}

\begin{Prop} \label{plethEg1}
$(-1)^{n+1}E_{n+1} = (1+q-N_1)(-1)^{n}E_n - q(-1)^{n-1}E_{n-1}$ for
$n \geq 2$.
\end{Prop}

\begin{proof}

One can prove this via plethysm, but it also follows directly from
the generating function for the $E_n$'s which is given by
$$\sum_{n\geq 0} (-1)^nE_nT^n = {(1-T)(1-qT)\over 1-(1+q-N_1)T+qT^2}.$$
The denominator of this series, also known as the series'
characteristic polynomial, yields the desired linear recurrence for
the coefficients of $T^{n+1}$, whenever $n+1$ exceeds the degree of
the
numerator. 
\end{proof}

With these two propositions verified, we can also now prove Theorem
\ref{EkasFib}.

\begin{proof} [Proof of Theorem \ref{EkasFib}]
It is clear that $E_1 = -F_{1}(q,-N_1)$, $E_2= F_3(q,-N_1)$, and
$E_3 = -F_5(q,-N_1)$.  Propositions \ref{FibRecurEk} and
\ref{plethEg1} show that both satisfy the same recurrence relations.
Thus we have verified that
$$E_n = (-1)^n
F_{2n-1}(q,-N_1).$$
\end{proof}

\begin{Rem}
We can utilize plethysm and obtain results of a similar flavor to
Proposition \ref{plethEg1}, for example see Lemma \ref{hka1a2}
below.  With this result in mind, we obtain the following table of
symmetric function $e_k$ and $h_k$ in terms of various alphabets. \\

\begin{tabular}{cccc}
poly. $\setminus$ alphabet& $1+q-\alpha_1-\alpha_2$ & $1+q$ & $\alpha_1+\alpha_2$ \\ \\

$e_k$ & $E_k$ & $e_1=1+q,~e_2= q$ & $e_1=1+q-N_1,~e_2= q$    \\ \\
$h_k$ & $H_k$ & $1+q+\dots +q^k$ &  $(-1)^kE_{k+1}/N_1$
\end{tabular}

\vspace{1em}\noindent Notice that the formulas for $e_k[1+q]$ and
$h_k[1+q]$ are precisely the $N_1=0$ cases of
$e_k[\alpha_1+\alpha_2]$ and $h_k[\alpha_1+\alpha_2]$. This should
come at no surprise since $1$ and $q$ are the two roots of
$T^2-(1+q)T+q$.
\end{Rem}

\begin{Lem} \label{hka1a2}
Letting $E_k$ be defined as $e_k[1+q-\alpha_1-\alpha_2]$ where
$\alpha_1$ and $\alpha_2$ are roots of $T^2-(1+q-N_1)T+q$, we obtain
$$h_k[\alpha_1+\alpha_2] = (-1)^kE_{k+1}/N_1.$$
\end{Lem}

\begin{proof} 
We have for $n \geq 2$ that
$$ N_1 E_n = E_{n+1} + (1+q)E_n + q E_{n-1}$$
since $(-1)^{n+1}E_{n+1} = (1+q-N_1)(-1)^{n}E_n -
q(-1)^{n-1}E_{n-1}$ by Proposition \ref{plethEg1}. However by
$$e_k[A-B] = \sum_{i=0}^k
e_{i}[A](-1)^{k-i}h_{k-i}[B],$$ we get $$E_{n+1} =
(-1)^{n+1}\bigg(h_{n+1}[\alpha_1+\alpha_2]-(1+q)h_n[\alpha_1+\alpha_2]
+ q h_{n-1}[\alpha_1+\alpha_2]\bigg)$$ using $A = 1+q$ and $B =
\alpha_1+\alpha_2$.  After verifying initial conditions and
comparing with $$ (-1)^{n+1}E_{n+1} = (-1)^{n+1}E_{n+2}/N_1
-(-1)^{n}(1+q)E_{n+1}/N_1 + (-1)^{n-1}q E_{n}/N_1$$ we get
$$h_{n+1}[\alpha_1+\alpha_2] = (-1)^{n+1}E_{n+2}/N_1$$ by induction.
\end{proof}

We apply the above $H_k$--$E_k$ (i.e. $h_k$--$e_k$) duality to
obtain an exponential generating function for the weighted number of
spanning trees of the wheel graph,
$$W(q,N_1,T) = \exp\bigg(\sum_{k\geq 1} \mathcal{W}_k(q,N_1)  {T^k
\over k} \bigg).$$ Using $\mathcal{W}_k=-N_k|_{N_1\rightarrow
-N_1}$, and the fact this is an exponential, we use (\ref{Zsimp}) to
obtain
$$W(q,N_1,T) = {1 \over 1-{N_1T\over (1-qT)(1-T)}} = {(1-qT)(1-T)\over 1- (1+q+N_1)T+qT^2}.$$

Also, rewriting $W(q,t,T)$
 as an ordinary generating
function, we get
$$W(q,t,T) = \sum_{k\geq0} E_k\bigg|_{N_1\rightarrow -N_1} (-T)^k = 1+
\sum_{k\geq 1} F_{2k-1}(q,t)T^k.$$

We summarize our results as the following dictionary between elliptic curves and spanning trees accordingly. \\

\begin{tabular}{ccc}
 & Elliptic Curves & Spanning Trees \\ \\

Generating Function & ${1-(1+q-N_1)T+qT^2 \over (1-qT)(1-T)}$ &
${(1-qT)(1-T)
\over 1-(1+q+N_1)T+qT^2}$ \\ \\
Factors of $1-(1+q\mp N_1)T+qT^2$ & $(1-\alpha_1T)(1-\alpha_2T)$ &
$(1-\beta_1T)(1-\beta_2T)$  \\ \\ $N_k ~ (resp.~  \mathcal{W}_k$ ) &
$p_k[1+q-\alpha_1-\alpha_2]$ & $p_k[-1-q+\beta_1+\beta_2]$ \\ \\
$H_k=N_1(1+q+\dots + q^{k-1})$ &  $h_k[1+q-\alpha_1-\alpha_2]$ & $(-1)^{k-1}e_k[-1-q+\beta_1+\beta_2]$   \\ \\
$(-1)^kE_k = F_{2k-1}(q,\mp N_1)$ & $(-1)^ke_k[1+q-\alpha_1-\alpha_2]$ &  $h_k[-1-q+\beta_1+\beta_2]$   \\ \\
\end{tabular}

\subsection{Duality between Lucas and Fibonacci numbers}
In addition to the above discussion of how $H_k$ and $E_k$ are dual,
this dictionary also highlights a comparison between \emph{elliptic
curve--spanning tree} duality and duality between Lucas numbers and
Fibonacci numbers.  As an application, we obtain a formula for
$E_k$, i.e. $F_{2k-1}(q,t)$, in terms of the polynomial expansion
for the $L_{2k}(q,t)$'s. If we recall our definition of $P_{i,k}$'s
such that $N_k = \sum_{i=1}^k (-1)^{i+1}P_{i,k}(q)N_1^i$, or
equivalently $L_{2k}(q,t) = 1+q^k + \sum_{i=1}^k P_{i,k}(q)t^i$,
then we have the following identity.

\begin{Prop} \label{EkAsFrac}
$$E_k = \sum_{i=1}^k {(-1)^{k+i} ~\cdot~ i\over k}~  P_{i,k}(q) N_1^i.$$
\end{Prop}

\begin{proof} We use the identities as above,
and the fact that ${1\over Z(E,T)} = \sum_{n\geq 0} (-1)^nE_n T^n$.
 Thus we have
\begin{eqnarray*}\sum_{n\geq 1} (-1)^nE_n T^n  &=& {1\over Z(E,T)}
~- 1 = {1\over 1+{N_1T\over (1-qT)(1-T)}} ~-1 = \sum_{n\geq 1}
(-1)^{n} \bigg({N_1T\over (1-qT)(1-T)}\bigg)^n
\\
&=& -N_1 {\partial \over \partial N_1}\sum_{n\geq 1}
{(-1)^{n-1}\over n} \bigg({N_1T\over (1-qT)(1-T)}\bigg)^n \\ &=&
-N_1 {\partial \over \partial N_1} \bigg(\log\bigg(1+{N_1T\over
(1-qT)(1-T)}\bigg)\bigg) = -N_1 {\partial \over \partial
N_1}\log\bigg(Z(E,T)\bigg),\end{eqnarray*}
which equals  $-N_1 {\partial \over \partial N_1}\bigg(\sum_{k\geq
1} {N_k\over k} T^k\bigg).$ Rewriting the $N_k$'s using the
polynomial formulas of Theorem \ref{Gar}, we have
\begin{eqnarray*} \sum_{n\geq 1} (-1)^nE_n T^n &=& -N_1 {\partial \over \partial
N_1}\bigg(\sum_{k\geq 1}{1\over k}\sum_{i=1}^k (-1)^{i-1}
P_{i,k}(q)N_1^i T^k\bigg)\\ &=& \sum_{k\geq 1}\sum_{i=1}^k {i\over
k}(-1)^{i} P_{i,k}(q)N_1^i T^k.
\end{eqnarray*}
Comparing the coefficients of $T^k$ on both sides completes the
proof.
\end{proof}

Proposition \ref{EkAsFrac} can also be given a combinatorial proof
by the following Lemma which contrasts the circular nature of our
combinatorial interpretation for the Lucas numbers with the linear
nature of the Fibonacci numbers.

\begin{Lem} \label{CijRels}
For $1\leq i \leq k$ and  $0\leq j \leq i$, we have the number,
which we denote as $c_{i,j}$, of subsets $S_1$ of $\{1,2,\dots,
2k\}$ with $k-i-j$ odd elements, $j$ even elements, and no two
elements circularly consecutive equals
$${k \over i} \cdot \# \bigg(\mathrm{subsets~}S_2\mathrm{~of~}\{1,2,\dots, 2k-2\}\mathrm{~with~}k-i-j\mathrm{~odd~elments,~}j\mathrm{~even~elements,}$$
\noindent $\mathrm{~and~no~two~elements~consecutive}\bigg).$
\end{Lem}
This notation might seem non-intuitive, but we use these indices so
that the total number of elements is $k-i$ and the number of even
elements is $j$. Thus the number of subsets $S_1$ (resp. $S_2$)
directly describes the coefficient of $q^jt^i$ in $L_{2k}(q,t)$
(resp. $F_{2k-1}(q,t)$).

\begin{proof}
To prove this result we note that there is a bijection between the
number of subsets of the first kind that do not contain $2k-1$ or
$2k$ and those of the second kind.  Thus it suffices to show that
the number of sets $S_1$ which \emph{do} contain element $2k-1$ or
$2k$ is precisely fraction ${k-i \over k}$ of all sets $S_1$
satisfying the above hypotheses.

Circularly shifting every element of set $S_1$ by an even amount
$r$, i.e. $\ell \mapsto \ell+r-1 ~(\mathrm{mod~} 2k) +1$, does not
affect the number of odd elements and even elements.  Furthermore,
out of the $k$ possible even shifts, $(k-i)$ of the sets, i.e. the
cardinality of set $S_1$, will contain $2k-1$ or $2k$.  This follows
since for a given element $\ell$ there is exactly one shift which
makes it $2k-1$ (or $2k$) if $\ell$ is odd (or even), respectively.
Since elements cannot be consecutive, there is no shift that sends
two different elements to both $2k-1$ and $2k$ simultaneously and
thus we get the full $(k-i)$ possible shifts.
\end{proof}

Using this relationship, we can derive formulas involving binomial
coefficients for $P_{i,k}(q)$ using our combinatorial interpretation
for the $(q,t)$-Lucas polynomials and $(q,t)$-Fibonacci polynomials.

\begin{Prop} \label{EkAsBinom}
For $k \geq 1$ and $1 \leq i \leq k$, we have
$$P_{i,k}(q) = \sum_{j=0}^i {k \over i}{k-1-j
\choose i-1}{i+j-1 \choose j} ~q^j.$$
\end{Prop}

\begin{proof}
See \cite[Theorem 2.2]{Zel} or \cite[Theorem 3]{MusPropp} which show
by algebraic and combinatorial arguments, respectively, that the
number of ways to choose a subset $S \subset \{1,2,\dots, 2n\}$ such
that $S$ contains $q$ odd elements, $r$ even elements, and no
consecutive elements is
$${n-r \choose q}{n-q \choose r}.$$  Letting $n=k-1$, $q=k-i-j$ and $r=j$, we
obtain
$$ {i\over k}P_{i,k}(q) = F_{2k-1}(q,N_1)\bigg|_{N_1^i} = \sum_{j=0}^i {k-1-j \choose i-1}{i+j-1 \choose j}q^j.$$
\end{proof}

\begin{Cor}
$$N_{k}(q,N_1) = \sum_{i=1}^k\sum_{j=0}^i {(-1)^{i+1}\cdot k \over i}{k-1-j
\choose i-1}{i+j-1 \choose j} N_1^i ~q^j.$$ and
$$E_k = \sum_{i=1}^k\sum_{j=0}^i (-1)^{k+i}~  {k-1-j \choose i-1}{i+j-1
\choose j} N_1^i ~q^j.$$
\end{Cor}

\begin{Rem}
From the proof in Section \ref{circtabloid}, we have that
\begin{eqnarray*} \mathcal{W}_k(q,N_1) &=& \sum_{\lambda \vdash k} {k \over
l(\lambda)}{l(\lambda) \choose d_1,~d_2,~\dots~d_r}
\bigg(\prod_{i=1}^{l(\lambda)} (1+q+q^2+\dots +
q^{\lambda_i-1})\bigg)N_1^{l(\lambda)}\\
&=& \sum_{i=1}^k {k \over i}\bigg(\sum_{\stackrel{\lambda\vdash
k}{l(\lambda)=i}}{ i\choose d_1,~d_2,~\dots~d_r} \prod_{j=1}^{i}
(1+q+q^2+\dots + q^{\lambda_j-1})\bigg)N_1^{i}
\end{eqnarray*}
which implies also that $$P_{i,k}(q) = {k \over
i}\sum_{\stackrel{\lambda\vdash k}{l(\lambda)=i}}{ i\choose
d_1,~d_2,~\dots~d_r} \prod_{j=1}^{i} (1+q+q^2+\dots +
q^{\lambda_j-1}).$$

\noindent Comparing the coefficients of this identity with the
coefficients in Proposition \ref{EkAsBinom} seems to give a
combinatorial identity that seems interesting in its own right.
\end{Rem}

\section{Factorizations of $N_k$}

We now introduce a family of $k$-by-$k$ matrices $M_k$ which, for
elliptic curves, yield a determinantal formula for $N_k$ in terms of
$q$ and $N_1$.

\begin{Thm}
\label{detformu}

Let $M_1 = \left[ -N_1\right]$, $M_2 = \left[\begin{matrix} 1+q-N_1
& -1-q \\ -1-q & 1+q-N_1\end{matrix}\right]$, and for $k\geq 3$, let
$M_k$ be the $k$-by-$k$ ``three-line'' circulant matrix
$$\left[
\begin{matrix}
1+q-N_1 & -1 & 0 &\dots & 0 &-q \\
-q & 1 +q -N_1 &  -1 & 0 & \dots& 0 \\
\dots & \dots & \dots & \dots & \dots & \dots\\
0 & \dots & -q & 1+q-N_1 & -1 & 0 \\
0 & \dots & 0 & -q & 1+q-N_1 & -1 \\
-1 & 0 & \dots & 0 & -q & 1+q-N_1
\end{matrix}\right].$$

The sequence of integers $N_k = \#C(\f_{q^k})$ satisfies the
relation $$N_k = -\det M_k \mathrm{~~for~all~}k\geq 1.$$
\end{Thm}
We provide two proofs of this theorem, one which
utilizes the three term recurrence from Section \ref{LucSec},
and one which introduces a new sequence of polynomials which are
interesting in their own right.

\subsection{Connection to orthogonal polynomials}

\label{chebsec}

Recall from the zeta function of an elliptic curve, $Z(E,T)$, we
derived a three term recurrence relation for the sequence $\{G_k =
1+q^k - N_k\}$:
\begin{eqnarray} \label{reccc}
G_{k+1} = (1+q-N_1)G_{k} - qG_{k-1}.\end{eqnarray} Such a relation
is indicative of an interpretation of the $(1+q^k-N_k)$'s as a
sequence of orthogonal polynomials.  In particular, any sequence of
orthogonal polynomials, $\{P_k(x)\}$, satisfies
\begin{eqnarray} \label{orthrec}
P_{k+1}(x) = (a_k x + b_k)P_k(x) + c_k P_{k-1}(x)
\end{eqnarray}
where $a_k$, $b_k$ and $c_k$ are constants that depend on $k \in
\n$.  Additionally, it is customary to initialize $P_{-k}(x) =0$,
$P_0(x) = 1,$ and $P_1(x) = a_0 x + b_0.$

Since we can think of the bivariate $N_k(q,N_1)$ as univariate
polynomials in variable $N_1$ with constants from field $\q(q)$, it
follows that recurrence (\ref{reccc}) is a special case of 
recurrence (\ref{orthrec}), therefore $\{P_k(x)\}_{k=1}^\infty = 
\{(1+q^k-N_k)(N_1)\}_{k=1}^\infty$ are a 
family of orthogonal polynomials.   In particular, we plug in the following 
values for the $a_k$, $b_k$, and $c_k$'s: 
\begin{align*}
a_k &= -1 \hspace{3em} \mathrm{~for~}k\geq 0 \\
b_k &= 1+q \hspace{2em} \mathrm{~for~}k\geq 0, \\ 
c_1
&= -2q \hspace{4em} \mathrm{and} \\ c_k &= -q \hspace{3em}
\mathrm{~for~}k\geq 2.
\end{align*}  (Note that we take $c_1$ to be $-2q$ since $G_0=1+q^0-N_0=2$, 
but we wish to normalize so that $P_0(x)=1$.)

In fact, the family $\{1+q^k-N_k\}_{k=1}^\infty$ can be described in terms of
a classical sequence of orthogonal polynomials.  
Namely $T_k(x)$ denotes the $k$th Chebyshev
(Tchebyshev) polynomials of the first kind, which are defined as
$\cos(k\theta)$ written out in terms of $x$ such that $\theta =
\arccos x$.
Equivalently, we can define $T_k(x)$ as the expansion of
$\alpha^k+\beta^k$ in terms of powers of $\cos \theta$ where
\begin{eqnarray*} \alpha &=& \cos \theta + i \sin \theta \\
\beta &=& \cos \theta - i \sin \theta. \end{eqnarray*}

\begin{Thm} \label{Chebk}
Considering the $(1+q^k-N_k)$'s as univariate polynomials in $N_1$
over the field $\q(q)$, we obtain
$$1+ q^k - N_k = 2q^{k/2}T_{k}\bigg((1+q-N_1)/2q^{1/2}\bigg).$$
\end{Thm}

\begin{proof}
We note that Chebyshev polynomials satisfy initial conditions
$T_0(x) = 1$, and $T_1(x)=x$ and the three-term recurrence
$$T_{k+1}(x) = 2xT_{k}(x) - T_{k-1}(x)$$ for $k\geq 1$ since
\begin{eqnarray*}
T_{k+1}(x) &=& \alpha^{k+1} + \beta^{k+1} \\
&=&
(\alpha+\beta)(\alpha^{k}+\beta^{k})-\alpha\beta(\alpha^{k-1}+\beta^{k-1})
\\ &=& 2\cos \theta ~T_{k}(x) - T_{k-1}(x)
\\ &=& 2xT_{k}(x) - T_{k-1}(x).
\end{eqnarray*}
Let $x = {1+q - N_1 \over 2\sqrt{q}}$. Clearly Theorem \ref{Chebk}
holds for $k=1$, and additionally, by Proposition \ref{Lrecur}, the
${1+q^{k}-N_{k}\over 2q^{k/2}}$'s satisfy the same recurrence as the
$T_k(x)$'s.  Namely
\begin{eqnarray*}{1+q^{k+1}-N_{k+1}\over 2q^{(k+1)/ 2}}   &=&
{(1+q-N_1)(1+q^k-N_k)-q(1+q^{k-1}-N_{k-1})\over 2q^{(k+1)/2}} \\&=&
2\bigg({1+q-N_{1}\over 2q^{1/2}} \bigg)\bigg({1+q^{k}-N_{k}\over
2q^{k/ 2}}\bigg) - \bigg({1+q^{k-1}-N_{k-1}\over 2q^{(k-1)/
2}}\bigg).
\end{eqnarray*}
\end{proof}

Another way to foresee the appearance of Chebyshev polynomials is by
noting that in the case that we plug in $q=0$ or $q=1$, we obtain a
family of univariate polynomials $\tilde{N}_k$ with the property
$\tilde{N}_{mk} = \tilde{N}_m(\tilde{N}_k) =
\tilde{N}_k(\tilde{N}_m)$.  It is a fundamental theorem of Chebyshev
polynomials that families of univariate polynomials with such a
property are very restrictive. In particular, from \cite{Cheb2} as
described on page $33$ of \cite{Cheb}:  If $\{\tilde{N}_k\}$ is a
sequence of integral univariate polynomials of degree $k$ with the
property
$$\tilde{N}_{mn} = \tilde{N}_m(\tilde{N}_n) =
\tilde{N}_n(\tilde{N}_m)$$ for all positive integers $m$ and $n$,
then $\tilde{N}_k$ must either be a linear transformation of
\begin{enumerate}
\item $x^k$ or

\item $T_k(x)$, the Chebyshev polynomial of the first
kind,
\end{enumerate}
where a linear transformation of a polynomial $f(x)$ is of the form
$$A\cdot f\bigg( (x-B)/A\bigg) + B \mathrm{~~or~equivalently~~} \bigg(f(\overline{A}x+\overline{B})- \overline{B}\bigg)\bigg/\overline{A}.$$

In particular we get formulas for $\mathcal{W}_k(0,N_1)$ and
$\mathcal{W}_k(1,N_1)$ (resp. $N_k(0,N_1)$ and $N_k(1,N_1)$) which
are indeed linear transformations of $x^k$ and $T_k(x)$
respectively.

\begin{Prop} \label{Chebforms}
\begin{eqnarray}N_k(0,N_1) &=& -(1-N_1)^k +1, \\
\label{pluginone} N_k(1,N_1) &=& -2T_k(-N_1/2+1) +2.
\end{eqnarray}
\end{Prop}

\begin{proof}
The coefficient of $N_1^m$ in $\mathcal{W}_k(0,N_1)$ is the number
of directed rooted spanning trees of $W_k$ with $m$ spokes and arcs
always directed counter-clockwise.  In particular it is only the
placement of the spokes that matter at this point since the
placement of the arcs is now forced.  Thus the coefficient of
$N_1^m$ in $\mathcal{W}_k(0,N_1)$ is ${k \choose m}$ for all $1 \leq
m \leq k$. Thus the generating function $\mathcal{W}_k(0,N_1)$
satisfies
$$\mathcal{W}_k(0,N_1) = (1+N_1)^k-1$$ since the constant term of
$\mathcal{W}_k(0,N_1)$ is zero.  Using the relation $N_k(q,N_1) =
-\mathcal{W}_k(q,-N_1)$ completes the proof in the $q=0$ case.  We
also note that $-(1-x)^k +1$ is a linear transformation of $x^k$ via
$A=-1$ and $B=1$.
The case for $q=1$ is a corollary of Theorem \ref{Chebk}.
\end{proof}

\subsection{First proof of Theorem \ref{detformu}: Using orthogonal polynomials}

As an application of Theorem \ref{Chebk}, we use the theory of orthogonal polynomials to learn properties of the $(1+q^k-N_k)$'s.
For example, one of the properties of a sequence of orthogonal polynomials is an interpretation as the determinants of a
family of tridiagonal $k$-by$k$ matrices.

\begin{Prop} \label{firmatrix}
$$1 + q^k - N_k = \det \left[ \begin{matrix}
 1+q-N_1         &  -2q        & 0         & 0      & 0      &     0 \cr
-1            &  1+q-N_1     &-q          & 0      & 0      & 0 \cr
0 & -1         & 1+q-N_1      &-q       & 0      &     0 \cr \vdots
& \vdots &\vdots     & \ddots & \ddots & 0     \cr 0 &0 &0 & \cdots
& 1+q-N_1 & -q \cr 0 & 0 &0 & \cdots & -1 & 1+q-N_1
\end{matrix}\right].$$  We denote this matrix as $M_k^\prime$.
\end{Prop}

\begin{proof}

Given a sequence of orthogonal polynomials satisfying $P_{0}(x) =
1$, $P_{1}(x) = a_0 x + b_0$ and recurrence (\ref{orthrec}), we have
the formula \cite{SchurSLC}
$$P_k(x) = \det \left[ \begin{matrix}
 a_0x + b_0         &  c_1        & 0         & 0      & 0      &     0 \cr
-1            &  a_1x+b_1     &c_2          & 0      & 0      & 0
\cr 0 & -1 &  a_2x+b_2     &c_3       & 0      &     0 \cr \vdots &
\vdots &\vdots & \ddots & \ddots & 0     \cr 0 &0 &0 & \cdots &
a_{k-2}x+b_{k-2} & c_k \cr 0 & 0 &0 & \cdots & -1 & a_{k-1}x+b_{k-2}
\end{matrix}\right].$$
Plugging in the $a_i$, $b_i$, and $c_i$'s as in Section $4.1$ yields the
formula.
\end{proof}

\begin{Rem}
Alternatively, we can use symmetric functions and the Newton
Identities \cite{EC2} to obtain these determinant identities, as
described in \cite[Chap. 7]{Gars} or \cite[Chap.
5]{MusThesis}.\end{Rem}

We can prove Theorem \ref{detformu} via Proposition \ref{firmatrix}
followed by an algebraic manipulation of matrix $M_k$.  Namely, by
using the multilinearity of the determinant, and expansions about
the first row followed by the first column, we obtain
$$\det(M_k) =\det(A_k) + \det(B_k) + \det(C_k) + \det(D_k)$$ where $A_k$, $B_k$, $C_k$,
and $D_k$ are the following $k$-by-$k$ matrices:

$$A_k =\left[ \begin{matrix}
 1+q-N_1         &  -1        & 0         & 0      & 0      &     0 \cr
-q            &  1+q-N_1     &-1          & 0      & 0      & 0 \cr
0 & -q         & 1+q-N_1      &-1       & 0      &     0 \cr \vdots
& \vdots &\vdots     & \ddots & \ddots & 0     \cr 0 &0 &0 & \cdots
& 1+q-N_1 & -1 \cr 0 & 0 &0 & \cdots & -q & 1+q-N_1
\end{matrix}\right].$$

$$B_k =\left[ \begin{matrix}
 0         &  0        & 0         & 0      & 0      &     -q \cr
-q            &  1+q-N_1     &-1          & 0      & 0      & 0 \cr
0 & -q         & 1+q-N_1      &-1       & 0      &     0 \cr \vdots
& \vdots &\vdots     & \ddots & \ddots & 0     \cr 0 &0 &0 & \cdots
& 1+q-N_1 & -1 \cr 0 & 0 &0 & \cdots & -q & 1+q-N_1
\end{matrix}\right].$$

$$C_k =\left[ \begin{matrix}
0         &  -1        & 0         & 0      & 0      &     0 \cr 0 &
1+q-N_1     &-1          & 0      & 0      & 0 \cr 0 & -q & 1+q-N_1
&-1       & 0      &     0 \cr \vdots & \vdots &\vdots & \ddots &
\ddots & 0     \cr 0 &0 &0 & \cdots & 1+q-N_1 & -1 \cr -1 & 0 &0 &
\cdots & -q & 1+q-N_1
\end{matrix}\right].$$

$$D_k =\left[ \begin{matrix}
 0         &  0        & 0         & 0      & 0      &     -q \cr
0            &  1+q-N_1     &-1          & 0      & 0      & 0 \cr 0
& -q         & 1+q-N_1      &-1       & 0      &     0 \cr \vdots &
\vdots &\vdots     & \ddots & \ddots & 0     \cr 0 &0 &0 & \cdots &
1+q-N_1 & -1 \cr -1 & 0 &0 & \cdots & -q & 1+q-N_1
\end{matrix}\right].$$
Cyclic permutation of the rows of $B_k$ and the columns of $C_k$
yield upper-triangular matrices with $-1$'s (resp. $-q$)'s on the
diagonal.  Given that the sign of such a cyclic permutation is
$(-1)^{k-1}$, we obtain $\det(B_k) + \det(C_k) = -q^k-1$.
Additionally, by expanding $\det(D_k)$ about the first row followed
by the first column, we obtain $\det(D_k) = -q\det(A_{k-2})$.  In
conclusion $$1+q^k + \det(M_k) = \det(A_k) - q\det(A_{k-2}).$$ After
transposing $M_k^\prime$, by analogous methods we obtain $$\det
M_k^\prime = \det(A_k) - q\det(A_{k-2})$$ and thus the desired
formula $\det M_k = -N_k$.

\subsection{Second proof of Theorem \ref{detformu}: Using the zeta function}

Alternatively, we  note that we can factor $$N_k = 1 + q^k -
\alpha_1^k - \alpha_2^k$$ using the fact that $q=\alpha_1\alpha_2$.
Consequently,
$$N_k = (1-\alpha_1^k)(1-\alpha_2^k)$$ and we can factor each of
these two terms using cyclotomic polynomials.  We recall that
$(1-x^k)$ factors as $$1-x^k=\prod_{d|k} Cyc_d(x)$$ where $Cyc_d(x)$
is a monic irreducible polynomial with integer coefficients.  We can
similarly factor $N_k$ as

$$N_k = \prod_{d|k} Cyc_d(\alpha_1)Cyc_d(\alpha_2).$$

These factors are therefore bivariate analogues of the cyclotomic
polynomials, and we refer to them henceforth as {\bf elliptic
cyclotomic polynomials},
denoted as $ECyc_d$.  

\begin{Def}
We define the elliptic cyclotomic polynomials to be a sequence of polynomials
in variables $q$ and $N_1$ such that for $d \geq 1$,
$$ECyc_d = Cyc_d(\alpha_1)Cyc_d(\alpha_2),$$ where $\alpha_1$ and $\alpha_2$ are
the two roots of $$T^2 - (1+q-N_1)T + q.$$  We verify that they can be
expressed in terms of $q$ and $N_1$ by the following proposition.
\end{Def}

\begin{Prop}
Writing down $ECyc_d$ in terms of $q$ and $N_1$
yields irreducible bivariate polynomials with integer coefficients.
\end{Prop}

\begin{proof}
Firstly we have $$\alpha_1^j + \alpha_2^j = (1 + q^j - N_j)  \in \z
$$ for all $j \geq 1$ and expanding a polynomial in $\alpha_1$ multiplied by the same polynomial in $\alpha_2$ yields terms of the form
$\alpha_1^i\alpha_2^i(\alpha_1^j+\alpha_2^j)$.  Secondly the
quantity $N_j$ is an integral polynomial in terms of $q$ and $N_1$
by Theorem \ref{Gar} and $\alpha_1^i\alpha_2^i = q^i$. Putting these
relations together, and the fact that $Cyc_d$ is an integral
polynomial itself, we obtain the desired expressions for $ECyc_d$.

Now let us assume that $ECyc_d$ is factored as $F(q,N_1)G(q,N_1)$.
The polynomial $Cyc_d(x)$ factors over the complex numbers as
$$Cyc_d(x) = \prod_{\stackrel{j=1}{\gcd(j,d)=1}}^d (1-\omega^jx)$$
where $\omega$ is a $d$th root of unity.  Thus $F(q,N_1) =
\prod_{i\in S} (1-\omega^i \alpha_1) \prod_{j\in T} (1-\omega^j
\alpha_2)$ for some nonempty subsets $S$, $T$ of elements relatively
prime to $d$.  The only way $F$ can be integral is if $F$ equals its
complex conjugate $\overline{F}$.  However, $\alpha_1$ and
$\alpha_2$ are complex conjugates by the Riemann hypothesis for
elliptic curves \cite{Hasse,Silver} (Hasse's Theorem), and thus
$F=\overline{F}$ implies that the sets $S$ and $T$ are equal.  Since
$Cyc_d(x)$ is known to be irreducible, the only possibility is $S =
T = \{j: \gcd(j,d)=1\}$, and thus $F(q,N_1) = ECyc_d$, $G(q,N_1)=1$.
\end{proof}

\begin{Rem} \label{fundsymm}
Alternatively, the integrality of the $ECyc_d$'s also follows from
the Fundamental Theorem of Symmetric Functions that states that a
symmetric polynomial with integer coefficients can be rewritten as
an integral polynomial in $e_1,e_2,\dots$.  In this case,
$Cyc_d(\alpha_1)Cyc_d(\alpha_2)$ is a symmetric polynomial in two
variables so $e_1=\alpha_1+\alpha_2=1+q-N_1$,
$e_2=\alpha_1\alpha_2=q$, and $e_k=0$ for all $k\geq 3$.  Thus we
obtain an expression for $ECyc_d$ as a polynomial in $q$ and $N_1$
with integer coefficients.
\end{Rem}

We can factor $N_k$, i.e. the $ECyc_d$'s even further, if we no
longer require our expressions to be integral.

\begin{eqnarray*}N_k &=& \prod_{j=1}^k (1-\alpha_1
\omega_k^j)(1-\alpha_2\omega_k^j) \\
\label{eigenECyc} &=& \prod_{j=1}^k (1- (\alpha_1+\alpha_2)\omega_k^j +
(\alpha_1\alpha_2)\omega_k^{2j}) \\
&=& (-1) \prod_{j=1}^k (-\omega_k^{k-j})(1- (1+q-N_1)\omega_k^j +
(q)\omega_k^{2j}) \\
&=& -\prod_{j=1}^k \bigg( (1+q-N_1) -
q\omega_k^{j} - \omega_k^{k-j}\bigg).
\end{eqnarray*}

Furthermore, the eigenvalues of a circulant matrix are well-known,
and involve roots of unity analogous to the expression precisely
given by the second equation above. (For example Loehr, Warrington,
and Wilf \cite{WWL} provide an analysis of a more general family of
three-line-circulant matrices from a combinatorial perspective.
Using their notation, our result can be stated as $$N_k =
\Phi_{k,2}(1+q-N_1,-q)$$ where $\Phi_{p,q}(x,y) = \prod_{j=1}^p (1-x
\omega^j - y \omega^{qj})$ and $\omega$ is a primitive $p$th root of
unity.  It is unclear how our combinatorial interpretation of $N_k$,
in terms of spanning trees, relates to theirs, which involves
permutation enumeration.) In particular, we prove Theorem
\ref{detformu} since $\det M_k$ equals the product of $M_k$'s
eigenvalues, which are precisely given as the $k$ factors of $-N_k$
in second equation above.

\subsection{Combinatorics of elliptic cyclotomic
polynomials}

In this subsection we further explore properties of elliptic
cyclotomic polynomials, noting that they are more than auxiliary
expressions that appear in the derivation of a proof.  To start
with, by M\"obius inversion, we can use the identity
\begin{eqnarray} \label{NktoECyc} N_k &=& \prod_{d|k}
ECyc_d(q,N_1)\end{eqnarray} to define elliptic cyclotomic polynomials directly
as

\begin{eqnarray}
\label{ECyctoNk} ECyc_k(q,N_1) &=& \prod_{d|k}
N_d^{~\mu(k/d)}\end{eqnarray} in addition to the alternative
definition \begin{eqnarray} ECyc_k(q,N_1) &=&
\prod_{\stackrel{j=1}{\gcd(j,d)=1}}^k \bigg( (1+q-N_1) -
q\omega_k^{j} - \omega_k^{k-j}\bigg).\end{eqnarray}

In particular, $ECyc_1 = N_1$ and $ECyc_p = N_p/N_1$ if $p$ is
prime.  To get a handle on $ECyc_k$ for $k$ composite, we provide
the following table for small values of $k$:

\begin{eqnarray*}
ECyc_4 &=& N_1^2 - (2+2q)N_1 + 2(1+q^2) \\
ECyc_6 &=& N_1^2 - (1+q)N_1 + (1 - q + q^2) \\
ECyc_8 &=& N_1^4 - (4+4q)N_1^3 + (6+8q+6q^2)N_1^2 -
(4+4q+4q^2+4q^3)N_1
+ 2(1+q^4) \\
ECyc_9 &=& N_1^6 - (6+6q)N_1^5 + (15+24q+15q^2)N_1^4 - (21+36q+36q^2+21q^3)N_1^3 \\
&+& (18+27q+27q^2+27q^3+18q^4)N_1^2 - (9+9q+9q^2+9q^3+9q^4+9q^5)N_1 + 3(1+q^3+q^6) \\
ECyc_{10} &=& N_1^4 - (3+3q)N_1^3 + (4+3q+4q^2)N_1^2 - (2+q+q^2+2q^3)N_1+ (1-q+q^2-q^3+q^4) \\
ECyc_{12} &=& N_1^4 - (4+4q)N_1^3 +
(5+8q+5q^2)N_1^2-(2+2q+2q^2+2q^3)N_1 + (1-q^2+q^4)
\end{eqnarray*}

We note several commonalities among these polynomials, as described in the
following propositions.  These properties are further rationale for our choice
of name for this family of polynomials.

\begin{Prop} \label{N=0,2q+2}
We have \begin{eqnarray}
ECyc_d|_{N_1=0} &=& C(d)Cyc_d(q) \\
\label{2q+2eq} ECyc_d|_{N_1=2q+2} &=& C^\prime(d)Cyc_d(-q)
\end{eqnarray} where $C(d)$ and $C^\prime(d)$ are the functions from $\z_{>0}$ to
$\z_{\geq 0}$ such that
$$C(d) =
\begin{cases}
0 \mathrm{~if~} d = 1 \\ p \mathrm{~if~} d = p^k \mathrm{~for~}p\mathrm{~prime} \\
1 \mathrm{~otherwise}
 \end{cases}$$
and
$$C^\prime(d) =
\begin{cases}
-2 \mathrm{~if~} d = 1 \\
~~0 \mathrm{~if~} d = 2 \\
~~p \mathrm{~if~} d = 2p^k \mathrm{~for~}p\mathrm{~prime~}(\mathrm{including~}2) \\
~~1 \mathrm{~otherwise}
 \end{cases}.$$
\end{Prop}

\begin{proof}
In the case that $N_1=0$, the characteristic quadratic equation
factors as $$1-(1+q-N_1)T + qT^2 = (1-T)(1-qT).$$  Consequently,
$\alpha_1=1$ and $\alpha_2=q$ in this special case.  (Note this is
strictly formal since $N_1=0$ is impossible, and thus it is not
contradictory that the Riemann Hypothesis fails.)  Nonetheless, we
still have $ECyc_d = Cyc_d(\alpha_1)Cyc_d(\alpha_2)$, and
consequently,
$$ECyc_d|_{N_1=0} = Cyc_d(1)Cyc_d(q).$$ Finally the value of
$Cyc_d(1)$ equals the function defined as $C(d)$ above \cite[Seq.
A020500]{intseq}.

For the reader's convenience we also provide a simple proof of this equality.
It is clear that $Cyc_1(q)=1-q$ and $Cyc_p(q)=1+q+q^2+\dots + q^{p-1}$ so by
induction on $k \geq 1$, assume that $Cyc_{p^k}(1)=p$.
$${1-q^{p^k} \over 1 -q} = 1 + q + q^2 + \dots + q^{p^k-1} = \prod_{j=1}^k
Cyc_{p^j}(q).$$  Plugging in $q=1$, and by induction we get $p^k =
p^{k-1} \cdot Cyc_{p^k}(1)$, thus we have $Cyc_{p^k}(1) =p$.  We now
proceed to show $Cyc_{d}(1) = 1$ if $d = p_1^{k_1}p_2^{k_2} \cdots
p_r^{k_r}$ for any $r\geq 2$.  For this we use $k$ such that $d|k$.
We assume $k = p_1^{k_1'}p_2^{k_2'} \cdots p_r^{k_r'}$.

\begin{eqnarray*}
{1-q^{k} \over 1 -q} &=& 1 + q + q^2 + \dots + q^{k-1}
\\ &=&
\bigg(\prod_{j_1=1}^{k_1'}
Cyc_{p_1^{j_1}}(q)\bigg)\bigg(\prod_{j_2=1}^{k_2'}
Cyc_{p_2^{j_2}}(q)\bigg)\cdots \bigg(\prod_{j_r=1}^{k_r'}
Cyc_{p_r^{j_r}}(q)\bigg)\\&\times&
\bigg(\prod_{d\mathrm{~is~another~divisor~of~}k}
Cyc_{d}(q)\bigg).\end{eqnarray*} The expression ${1-q^{k} \over 1
-q}\bigg|_{q=1}$ equals $k$, and the first $r$ products on the
right-hand-side equal $p_1^{k_1'},~p_2^{k_2'},~ \dots, ~p_r^{k_r'}$
respectively. Thus the last set of factors, i.e. the cyclotomic
polynomials of $d$ with two or more prime factors, must all equal
the value $1$.

We prove (\ref{2q+2eq}) analogously.  When $N_1=2q+2$ (again this is
strictly formal), the characteristic equation factors as
$$1-(1+q-N_1)T + qT^2 = (1+T)(1+qT)$$ implying $\alpha_1=-1$ and
$\alpha_2=-q$. Additionally, $C^\prime(d) = Cyc_d(-1)$ was observed
by Ola Veshta on Jun 01 2001, as cited on \cite[Seq.
A020513]{intseq}.
\end{proof}

\begin{Prop}
For $d\geq 2$, $$\deg_{N_1} ECyc_d = \deg_q ECyc_d = \phi(d),$$
where the Euler
$\phi$ function which counts the number of integers between $1$ and
$d-1$ which are relatively prime to $d$.
\end{Prop}

\begin{proof}
As noted in Remark \ref{fundsymm}, we can write $ECyc_d$ as an
integral polynomial in $e_1 = \alpha_1 +\alpha_2 = 1+q-N_1$ and
$e_2=\alpha_1\alpha_2=q$. The highest degree of $N_1$ in $ECyc_d$ is
therefore equal to the highest degree of $e_1=\alpha_1+\alpha_2$,
which is the same as the largest $m$ such that
$\alpha_1^m\alpha_2^0$ (resp. $\alpha_1^0\alpha_2^m)$ is a term in
$Cyc_d(\alpha_1)Cyc_d(\alpha_2)$.  Thus $\deg_{N_1} ECyc_d(q,N_1) =
\deg_{\alpha_1} Cyc_d(\alpha_1) = \phi(d)$.
Analogously, the degree of $q$ comes from the highest power of $(\alpha_1\alpha_2)^m$ in
$Cyc_d(\alpha_1)Cyc_d(\alpha_2)$.  Thus we have shown
$$\deg_q ECyc_d \leq \phi(d).$$  Equality follows from the first half of
Proposition \ref{N=0,2q+2} when $d\geq 2$ since the constant term with respect to $N_1$, which equals
$C(d)Cyc_d(q)$, has degree $\phi(d)$.
\end{proof}

Finally, if one examines the expressions for $ECyc_d(q,N_1)$, one
notes that they appear alternating in sign just as the polynomials
for $N_k$, except for the constant term which equals $C(d)Cyc_d(q)$
by Proposition \ref{N=0,2q+2}.  More precisely, the author finds the
following empirical evidence for such a claim.

\begin{Prop} \label{ellalt} For $d$ between $2$ and  $104$, we obtain
$$ECyc_d(q,N_1) = Cyc_d(1)\cdot Cyc_d(q) + \sum_{i=1}^{\phi(d)}
(-1)^{i} Q_{i,d}(q)N_1^i$$ where $Q_{i,d}$ is a univariate
polynomial with \emph{positive} integer coefficients.
\end{Prop}

However, the conjecture fails for $d=105$.  In particular if we
write
\begin{eqnarray*} ECyc_{105}(q,N_1) &=& Cyc_{105}(1)\cdot Cyc_{105}(q) +
\sum_{i=1}^{48} (-1)^{i} Q_{i,105}(q)N_1^i
\end{eqnarray*} where the $Q_{i,105}(q)$'s are univariate polynomials with integer coefficients, then
$Q_{2,105}(q)$ through $Q_{48,105}(q)$ indeed have \emph{positive}
integer coefficients as expected.  However the first univariate
polynomial, i.e. the coefficient of $-N_1$ is
\begin{eqnarray*} Q_{1,105}(q) &=& 24q^{47}+47q^{46}+69q^{45}+69q^{44}+69q^{43}+50q^{42}+32q^{41}
\\
&-& 2q^{40}-18q^{39}-33q^{38}-33q^{37}-33q^{36}-21q^{35}-10q^{34}
\\ &+&
9q^{32}+17q^{31}+24q^{30}+24q^{29}+24q^{28}+20q^{27}+20q^{26}+18q^{25}+18q^{24}
\\ &+&
18q^{23}+18q^{22}+20q^{21}+20q^{20}+24q^{19}+24q^{18}+24q^{17}+17q^{16}+9q^{15}
\\ &-& 
10q^{13}-21q^{12}-33q^{11}-33q^{10}-33q^{9}-18q^{8}-2q^{7}
\\ &+& 32q^6+50q^5+69q^4+69q^3+69q^2+47q+24.
\end{eqnarray*}
Note that there are $46$ nonzero coefficients of $Q_{1,105}$ in the
expansion of $ECyc_{105}(q,N_1)$, $14$ of which have the incorrect
sign.

The number $105 = 3\cdot 5 \cdot 7$ is significant and interesting
from a number theoretic point of view.  This number is also the
first $d$ such that ordinary cyclotomic polynomial $Cyc_d$ has a
coefficient other than $-1,0,$ or $1$.

\vspace{-0.7em}
\begin{eqnarray*} Cyc_{105} &=& 1 + x + x^2 - x^5 - x^6 - 2x^7 - x^8 - x^9 + x^{12} + x^{13} + x^{14}
\\&+& x^{15} + x^{16} + x^{17} - x^{20} - x^{22}
-x^{24} - x^{26} - x^{28} + x^{31} + x^{32} \\
&+& x^{33} + x^{34} + x^{35} + x^{36} - x^{39} - x^{40} - 2x^{41} -
x^{42} - x^{43} \\&+& x^{46} + x^{47} + x^{48}.
\end{eqnarray*}

Despite this counter-example, we still can prove that the
coefficients of the $ECyc_d$'s alternate in sign for an infinite
number of $d$'s.  Specifically, we note that $ECyc_{2^m}$ resemble
the coefficients of $N_{2^{m-1}}$, and moreover the pattern we find
is given by the following proposition.
\begin{Prop} \label{ECyc2m}
\begin{eqnarray} \label{QP2ell} ECyc_{2^m} = 2Cyc_{2^{m-1}}(q) - N_{2^{m-1}}.\end{eqnarray} In
particular, for $i$ between $1$ and $\phi(2^m)=2^{m-1}$, we get
\begin{eqnarray} \label{QP2} Q_{i,2^{m}} =
P_{i,2^{m-1}}\end{eqnarray} where the $P_{i,k}$ are the coefficients
of $N_k$.
\end{Prop}

Note that in our proof we use the fact that $ECyc_d$ can be written
as
$$Cyc_d(1)\cdot Cyc_d(q) + \sum_{i=1}^{\phi(d)} (-1)^{i}
Q_{i,d}(q)N_1^i$$ where the $Q_{i,d}$'s are univariate polynomials
with \emph{possibly} negative coefficients.  Therefore, our proof of
Proposition \ref{ECyc2m} actually extends Proposition \ref{ellalt}
to the case where $d$ is a power of $2$ since we previously showed
that the $P_{i,d}$'s alternate.

\begin{proof}
We note that $Cyc_{2^{m-1}} = 1 + q^{2^{m-1}}$ and that (\ref{QP2})
follows from (\ref{QP2ell}).  Also, $ECyc_{2^m} =
N_{2^m}/N_{2^{m-1}}$ and thus it suffices to prove
\begin{eqnarray*}N_{2^m} = (2+ 2q^{2^{m-1}})N_{2^{m-1}} - {N_{2^{m-1}}^2}.\end{eqnarray*}
However, this is a special case of $$N_2(q,N_1) = (2 + 2q)N_1(q,N_1)
-N_1(q,N_1)^2$$ where we plug in $q^{2^{m-1}}$ in the place of $q$.
\end{proof}

Unfortunately, formulas for $Q_{i,d}$'s in terms of $P_{i,k}$'s when
$d$ is not a power of $2$ are not as simple.  On the other hand, the
last part of this proof highlights a principle that has the
potential to open up a new direction.  Namely, $N_k(q,N_1)$ is
defined as the number of points on $C(\f_{q^k})$ where $q$ itself
can also be a power of $p$. Consequently, \begin{eqnarray}N_{m\cdot
k}(q, ~N_1) = \# C(\f_{q^{m\cdot k}}) = N_{m}\bigg(q^{k},
~N_{k}\bigg).\end{eqnarray} While this relation is immediate given
our definition of $N_k = \#C(\f_{q^k})$, when we translate this
relation in terms of spanning trees, the relation
\begin{eqnarray} \label{WmkRel} \mathcal{W}_{mk}(q,t) =
\mathcal{W}_m\bigg(q^k,\mathcal{W}_k(q,t)\bigg)\end{eqnarray} seems
much more novel. Furthermore, in this case, this relation involves
only positive integer coefficients and thus motivates exploration
for a bijective proof.  As noted in Section \ref{chebsec}, such a
compositional formula is indicative of the appearance of a linear
transformation of $x^k$ or $T_k(x)$, which is also clear from the
three-term recurrence satisfied by the $(1+q^k-N_k)$'s.

\subsection{Geometric interpretation of elliptic cyclotomic
polynomials}

Despite the fact that the above expressions of elliptic cyclotomic
polynomials do not have positive coefficients nor coefficients with
alternating signs, we can nonetheless describe a set of geometric
objects which the elliptic cyclotomic polynomials enumerate.

\begin{Thm} \label{geomEcyc} We have
$$ECyc_d = \bigg|Ker\bigg(Cyc_d(\pi)\bigg): ~ C(\overline{\f_q}) \rightarrow C(\overline{\f_q}) \bigg|$$ where
$\pi$ denotes the Frobenius map, and $Cyc_d(\pi)$ is an element of
$End(C)=End(C(\overline{\f_q}))$.
\end{Thm}

\begin{proof}
One of the key properties of the Frobenius map is the fact that
$C(\f_{q^k}) = Ker(1-\pi^k)$, where $1-\pi^k$ is an element of
$End(C)$.  See \cite{Silver} for example.  The map $(1-\pi^k)$
factors into cyclotomic polynomials in $End(C)$ since the
endomorphism ring contains both integers and powers of $\pi$.
Since the maps $Cyc_d(\pi)$ are each group homomorphisms, it follows
that the cardinality of $\bigg|\mathrm{Ker~}\bigg(
Cyc_{d_1}Cyc_{d_2}(\pi)\bigg)\bigg|$ equals  $\bigg|\mathrm{Ker~}
Cyc_{d_1}(\pi)\bigg|\cdot \bigg|\mathrm{Ker~} Cyc_{d_2}(\pi)\bigg|$
.
Thus
\begin{eqnarray*} 
\prod_{d|k} ECyc_d = N_k = \bigg|\mathrm{Ker}~(1-\pi^k)\bigg| =
\bigg|\mathrm{Ker} \prod_{d|k} Cyc_d(\pi)\bigg| = \prod_{d|k}
\bigg|\mathrm{Ker} ~ Cyc_d(\pi)\bigg|,    \end{eqnarray*} and since
the last equation is true for all $k \geq 1$, we must have the
relations
\begin{eqnarray} \label{ellcyc}
ECyc_d = \bigg|\mathrm{Ker} ~ Cyc_d(\pi)\bigg|
\end{eqnarray}
for all $d \geq 1$.
\end{proof}

Since $$N_k = \prod_{d|k} ECyc_d(q,N_1)$$ and
$\mathcal{W}_k(q,t)=-N_k\bigg|_{N_1\rightarrow -t}$, it also makes
sense to consider the decomposition
$$\mathcal{W}_k(q,t) = \prod_{d|k} WCyc_d(q,t)$$ where $WCyc_d(q,t) = -ECyc_d|_{N_1 \rightarrow -t}$.

This motivates the analogous question, namely does there exist a
combinatorial or geometric interpretation of these polynomials? We
in fact can answer this in the affirmative and do so in \cite[Chap.
6]{MusThesis} as well as in a forthcoming paper.

\begin{Rem}
The coefficients of the $WCyc_d$'s are always integers, but not
necessarily positive, as seen in the constant coefficient, as well
as in the counter-example $WCyc_{105}$.  Nonetheless, plugging in
specific integers $q\geq 0$ and $t\geq 1$ do in fact result in
positive expressions, which factor $\mathcal{W}_k(q,t)$.  It is
these values that we are interested in understanding.
\end{Rem}

\section{Conclusions and open problems}

The new combinatorial formula for $N_k$ presented in this write-up
appears fruitful.  It leads one to ask how spanning trees of the
wheel graph are related to points on elliptic curves.
For instance, is there a reciprocity that explains combinatorially
why the bivariate integral polynomial formulas for counting points
on elliptic curves and counting spanning trees of the wheel graph
are equivalent except for the appearance of alternating signs?  Such
reciprocities occur frequently in combinatorics.  For example given
the chromatic polynomial $\chi(\lambda)$ of a graph $G=(V,E)$, the
expression $(-1)^{|V|}\chi(-1)$ provides a formula for the number of
acyclic orientations of $G$ \cite{Chrom}.

The fact that the Fibonacci and Lucas numbers also enter the picture
is also exciting since these numbers have so many
different combinatorial interpretations, and there is such an
extensive literature about them.  Perhaps these combinatorial
interpretations will lend insight into why $N_k$ depends only on the
finite data of $N_1$ and $q$ for an elliptic curve, and how we can
associate points over higher extension fields to points on
$C(\f_q)$.

The elliptic cyclotomic polynomials provide an additional source of
new questions. What is the spanning tree interpretation of
$\mathcal{W}_k(q,N_1)$'s factorization?  Is there a combinatorial
interpretation of $\mathcal{W}_{mk}(q,t) =
\mathcal{W}_m(q^k,\mathcal{W}_k(q,t))$? What is a combinatorial
interpretation of the integral polynomials $Q_{i,d}$, and what does
the fact their coefficients are almost all positive mean? We will
tackle some of these problems in a forthcoming paper in which we
compare more thoroughly the structures of elliptic curves and
spanning trees.

\vspace{2em} \noindent \bf Acknowledgements. \rm The author would
like to thank Adriano Garsia for many useful conversations and his
invaluable guidance through the author's graduate school.  I would
also like to thank the referees for their valuable reports in leading
to a clearer exposition. In particular, the connection to orthogonal
polynomials, including the statement of Theorem $6$, was indicated
by one of the referees, as were the observations that several of the
propositions have direct algebraic proofs. Additionally, Thomas
Shemanske brought the significance of the number $105$, with respect
to ordinary cyclotomic polynomials, to our attention. Section $2$ of
this paper was presented at FPSAC $2006$ and the author would like
to thank the conference referees for their edits. This work was
supported by the NSF, grant DMS-0500557.

\bibliographystyle{amsalpha}

\end{document}